\documentclass[10pt]{amsart}
\usepackage{fullpage}
\usepackage{graphics}
\usepackage{amsmath,amsfonts,amssymb,amsthm,verbatim,shuffle}
\usepackage{bbold}

\usepackage[nobysame]{amsrefs}
\usepackage[usesnames,svgnames]{xcolor}
\usepackage{tikz}
\usepackage[enableskew,vcentermath]{youngtab}
\usepackage{ytableau,mathdots,yhmath}
\usepackage[all]{xy}
\usepackage{enumerate}
\usepackage[colorinlistoftodos]{todonotes}
\newtheorem{theorem}{Theorem}[section]
\newtheorem{lemma}[theorem]{Lemma}
\newtheorem{conjecture}[theorem]{Conjecture}
\newtheorem{question}[theorem]{Question}
\newtheorem{proposition}[theorem]{Proposition}

\newtheorem{definition}[theorem]{Definition}
\newtheorem{remark}[theorem]{Remark}
\newtheorem{example}[theorem]{Example}

\newcommand{\inv}[1]{\text{inv}(#1)}

\DeclareMathOperator{\FQSym}{\mathrm{FQSym}}
\DeclareMathOperator{\QSym}{\mathrm{QSym}}

\newcommand\mapsfrom{\mathrel{\reflectbox{\ensuremath{\longmapsto}}}}

\newcommand{\Inv}{{\mathrm{Inv}}}

\newlength{\cellsize} 
\newsavebox{\cell}
\sbox{\cell}{\begin{picture}(18,18)
\put(0,0){\line(1,0){18}}
\put(0,0){\line(0,1){18}}
\put(18,0){\line(0,1){18}}
\put(0,18){\line(1,0){18}}
\end{picture}}
\newcommand\cellify[1]{\def\thearg{#1}\def\nothing{}%
\ifx\thearg\nothing
\vrule width0pt height\cellsize depth0pt\else
\hbox to 0pt{\usebox{\cell} \hss}\fi%
\vbox to \cellsize{
\vss
\hbox to \cellsize{\hss$#1$\hss}
\vss}}
\newcommand\tableau[1]{\vtop{\let\\\cr
\baselineskip -16000pt \lineskiplimit 16000pt \lineskip 0pt
\ialign{&\cellify{##}\cr#1\crcr}}}

\newcommand{\pic}{\begin{tikzpicture}}
\newcommand{\epic}{\end{tikzpicture}}

\newcommand{\R}{\mathbb{R}}
\newcommand{\Q}{\mathbb{Q}}

\newcommand{\Z}{\mathbb{Z}}

\newcommand{\Des}{\mathrm{Des}}

\newcommand{\QSYM}{\mbox{QSYM}}

\newcommand{\bubble}{\times}

\newcommand{\one}{\mathbb{1}}

\newcommand{\Alt}{\mathrm{Alt}}
\newcommand{\comp}{\mathrm{comp}}
\newcommand{\interlace}{\mathrm{lace}}

\newcommand{\HHH}{\mathcal{H}}

\newcommand{\ASM}{\mathrm{ASM}}
\newcommand{\MT}{\mathrm{MT}}
\newcommand{\symm}{\mathfrak{S}}

\title{Weak order and descents for monotone triangles}

\author{Zachary Hamaker}
\address{Department of Mathematics\\
University of Michigan\\
Ann Arbor, MI 48109}
\email{zachary.hamaker@gmail.com}

\author{Victor Reiner}
\address{School of Mathematics\\
University of Minnesota\\
Minneapolis, MN 55455}
\email{reiner@umn.edu}


\begin{document}

\abstract{
Monotone triangles are a rich extension of permutations that biject with alternating sign matrices.
The notions of weak order and descent sets for permutations are generalized here to monotone triangles, and shown to enjoy many analogous properties.  It is shown that any linear extension of the weak order gives rise to a shelling
order on a poset, recently introduced by Terwilliger, whose maximal chains biject with monotone triangles;  among
these shellings are a family of EL-shellings.

The weak order turns out to encode an action of the $0$-Hecke monoid of type $A$ on the monotone triangles,
generalizing the usual bubble-sorting action on permutations.
It also leads to a notion of descent set for monotone triangles, having another natural property:
the surjective algebra map from the Malvenuto-Reutenauer Hopf algebra of permutations into quasisymmetric functions 
extends in a natural way to an algebra map out of the recently-defined Cheballah-Giraudo-Maurice algebra of alternating sign matrices.}
}

\maketitle


\section{Introduction}
\label{s:intro}

Permutations in the symmetric group $\symm_n$
on $n$ letters, when thought of as $n \times n$ permutation matrices,
are special cases of fascinating objects known as
{\it alternating sign matrices} ({\it ASMs} ).   The latter have been intensely studied since their introduction by Mills, Robbins and Rumsey \cite{mills1983alternating}, and turn out to be connected with such areas as statistical mechanics, representation theory, and number theory--
see Bressoud~\cite{bressoud1999proofs} and Brubaker, Bump and Friedberg \cite{BrubakerBumpFriedberg}
for more history and context.  We recall their definition here,
as well as their bijection with the equivalent objects known as {\it monotone triangles}.

A vector in $\{0,\pm 1\}^n$ is called {\it alternating} if
its $\pm 1$ values alternate in sign, beginning and ending with $+1$.  Denote by
$\Alt_n$ the set of all such alternating vectors of length $n$.
An $n \times n$ \emph{alternating sign matrix} is one whose
row and column vectors all lie in $\Alt_n$.  Denote by $\ASM_n$ the set of all such
matrices.  For example, we depict here on the left a matrix $A$ in $\ASM_6$, 
abbreviating $"+"$ and $-$ for entries $+1$ and $-1$:
\begin{equation}
\label{eq:asm}
\left[
\begin{smallmatrix}
0 &+ & 0 & 0 & 0 & 0 \\
0 & 0 & 0 &+ & 0 & 0 \\
+&- &+ &- & 0 &+ \\
0 & 0 & 0 &+ & 0 & 0 \\
0 &+ & 0 &- &+ & 0 \\
0 & 0 & 0 & + &0 & 0 
\end{smallmatrix}
\right]=A
\qquad \qquad \leftrightarrow \qquad\qquad
T=\xymatrix@R=.01in@C=.01in{
*\txt<5pc>{2\\2 4\\1 3 6\\1 3 4 6\\1 2 3 5 6}}
\end{equation}
There is a simple bijection between $\ASM_n$ and the set $\MT_n$ of
{\it monotone triangles of size $n$}.  A monotone triangle of size $n$ is a
sequence $T=(T_0,T_1,\ldots,T_{n-1},T_n)$ of subsets of
$[n]:=\{1,2,\ldots,n\}$ where $\#T_m=m$, with the extra property that 
$T_{m+1}$ {\it interlaces} $T_m$ in this sense:  if  one list entries of $T_m, T_{m+1}$ in increasing order as
$$
\begin{array}{rcll}
T_m&=&\{i_1 < i_2 < \cdots < i_m\}, \\
T_{m+1}&=&\{j_1 < j_2 < \cdots < j_m < j_{m+1}\}, 
\end{array}
$$
then one has
\begin{equation}
\label{classical-interlacing-definition}
j_1 \leq i_1 \leq j_2 \leq i_2 \leq j_3 \leq \cdots \leq j_m \leq i_m \leq j_{m+1}.
\end{equation}
One depicts $T$ as a triangular array having $T_m$ as its $m^{th}$ row
from the top, omitting $T_0=\varnothing, T_n=[n]$. For example, 
$T=( \varnothing, \{2\}, \{2,4\}, \{1,3,6\} , \{1,3,4,6\} , \{1,2,3,5,6\} , [6] ) \in \MT_6$ is shown on the right  in \eqref{eq:asm}.
For the sake of defining the bijections $\ASM_n \leftrightarrow \MT_n$,
first introduce the {\it indicator vector} $\one_S$ in 
$\{0,1\}^n$ for a subset $S \subseteq [n]$, having coordinates $(\one_{S})_{i}=1$ for $i \in S$ and $(\one_{S})_{i}=0$ for $i \not\in S$.  Then 
given $A$ in $\ASM_n$, one maps $A \mapsto T=(T_0,\ldots,T_n)$ in $\MT_n$ whose $m^{th}$ row 
$T_m$ is the unique subset for which $\one_{T_m}$ is the sum of the first $m$ rows of $A$.  The inverse bijection sends $T \mapsto A$
where the $m^{th}$ row of $A$ is $\one_{T_{m}}-\one_{T_{m-1}}$. For example, the matrix $A$ in $\ASM_6$ shown on the left in \eqref{eq:asm} above has corresponding monotone triangle $T$ in $\MT_6$ shown to its right.

It is not hard to check (see Terwilliger~\cite[Thm. 3.2]{terwilliger2017poset}) that 
an $(m+1)$-subset $J\subset [n]$ interlaces an $m$-set $I \subset [n]$ if and only if the difference of the indicator vectors
$\one_J-\one_I$ lies in $\Alt_n$.  Thus $\MT_n$ is in bijection with the maximal chains of a partial order on the subsets of  $[n]$ that is the transitive
closure of the relation $I < J$ when $J$ interlaces $I$;  Terwilliger denotes this partial order $\Phi_n$.
Note that this partial order $\Phi_n$ is stronger than the usual {\it Boolean algebra poset $2^{[n]}$}, whose order relation is given
by inclusion $\subseteq$, and whose maximal chains are the monotone triangles of the form 
$
T(w):=(\varnothing, \{w_1\},\{w_1,w_2\},\ldots,\{w_1,w_2,\ldots,w_{n-1}\},[n]),
$
which correspond to the permutations $w=(w_1,w_2,\ldots,w_n)$ in $\symm_n$.  This monotone
triangle $T(w)$ also corresponds to the usual permutation matrix of $w^{-1}$, thinking of permutation matrices as a subset of $\ASM_n$.
The Hasse diagram for the poset $\Phi_3$ on subsets of $[3]$ is shown below, with solid edges indicating the weaker Boolean algebra $2^{[3]}$ ordering, and the unique extra order relation $\{2\} < \{1,3\}$ from $\Phi_3$ shown dotted: 
$$
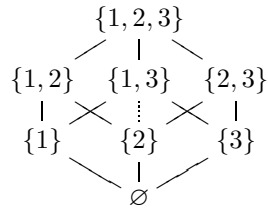
\xymatrix@R=.1in@C=.02in{
       & \{1,2,3\} & \\
\{1,2\}\ar@{-}[ur]& \{1,3\}\ar@{-}[u]& \{2,3\}\ar@{-}[ul]\\
\{1\}\ar@{-}[ur]\ar@{-}[u]& \{2\}\ar@{-}[ur]\ar@{.}[u]\ar@{-}[ul]& \{3\}\ar@{-}[ul]\ar@{-}[u]\\
     &\varnothing\ar@{-}[ur]\ar@{-}[u]\ar@{-}[ul]&
}
$$
Section~\ref{s:interlacing} explores properties of the order $\Phi_n$, including characterizing it via
a generalization of interlacing.  

One of our original goals was to show that $\Phi_n$ is a {\it shellable} poset, a notion that we review
here.  Say that an abstract simplicial complex $\Delta$ is {\it pure} if all of its {\it facets} (=inclusion-maximal simplices)
have the same number of vertices.  In this case, say that an
ordering $F_1,F_2,\ldots $ of the facets of $\Delta$ is a (pure) {\it shelling} if for every $j \geq 2$, the 
intersection of the boundary of $F_j$ with the subcomplex generated by the facets $F_1,\ldots,F_{j-1}$ forms a pure subcomplex
of codimension one within the boundary of $F_j$; said differently, for any pair $1 \leq i<j$ , there exists $k<j$ such
that $F_i \cap F_j \subseteq F_k \cap F_j$ with $\# F_k \cap F_j=\#F_j-1$.  Having a shelling for $\Delta$ imposes strong topological properties for its {\it geometric realization} $\Vert \Delta \Vert$, and strong algebraic properties for its {\it Stanley-Reisner ring} $k[\Delta]$; see Bj\"orner \cite[Appendix]{bjorner1980shellable} and \cite[\S 1]{bjorner1984some}.  
Here we are starting with a partially ordered set $P$ having both a bottom element $\hat{0}$ and top element $\hat{1}$, such as the {\it Boolean algebra $2^{[n]}$} with inclusion order on subsets of $[n]$, or the order $\Phi_n$ on subsets, where in either case, $\hat{0}=\varnothing$ and $\hat{1}=[n]$.  In this setting, one often removes the bottom and top elements, and associates an abstract simplicial complex called the {\it order complex} to its {\it proper part},  so that 
$\Delta$ has vertex set $P \setminus \{\hat{0},\hat{1}\}$, and simplices
for each totally ordered subset of $P \setminus \{\hat{0},\hat{1}\}$. This means that facets of $\Delta$ biject with maximal chains of $P$.  

As mentioned above, for $P=\Phi_n$ and its subposet the Boolean algebra $2^{[n]}$, these facets or maximal chains are naturally labeled by the monotone triangles $\MT_n$ and permutations $\symm_n$, respectively.  We illustrate this here for $n=3$, depicting the order complex $\Delta(\Phi_3 \setminus \{\hat{0},\hat{1}\})$, with one extra facet (edge) shown
dotted, whose removal gives the subcomplex $\Delta(2^{[3]} \setminus \{\hat{0},\hat{1}\})$.
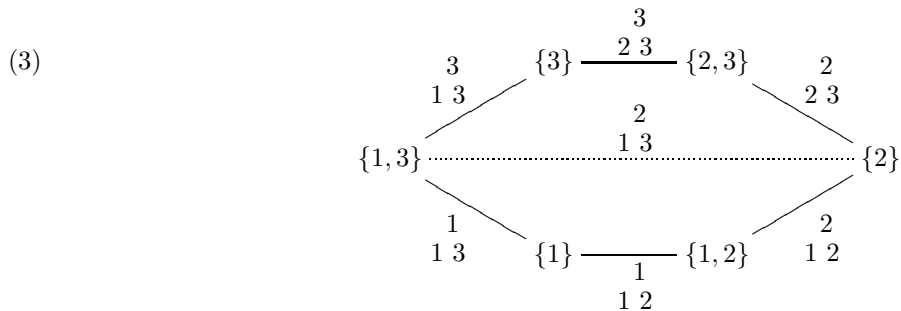
\begin{equation}
\label{simplicial-complexes-for-n=3}
\xymatrix@R=.1in@C=.5in{
        & \{3\}\ar@{-}^{\txt<2pc>{ 3 \\2 3}}[r]& \{2,3\}\ar@{-}^{\txt<2pc>{ 2 \\2 3}}[ddr]& \\
 & &  &\\
\{1,3\}\ar@{.}^{\txt<2pc>{ 2 \\1 3}}[rrr] \ar@{-}^{\txt<2pc>{ 3 \\1 3}}[uur]\ar@{-}_{\txt<2pc>{ 1 \\1 3}}[ddr]&        &        &\{2\}\\
 & & & \\
        & \{1\}\ar@{-}_{\txt<2pc>{ 1 \\1 2}}[r]  & \{1,2\}\ar@{-}_{\txt<2pc>{ 2 \\1 2}}[uur]&
}
\end{equation}
For the Boolean algebra $2^{[n]}$, this order complex $\Delta(2^{[n]} \setminus \{\hat{0},\hat{1}\})$ is isomorphic to the {\it Coxeter 
complex} of type $A_{n-1}$, and a result of Bj\"orner \cite[Thm. 2.1]{bjorner1984some} shows that it is shellable, 
with a shelling order on its facets provided by
any linear ordering on the permutations $\symm_n$ that extends the {\it (right) weak order} $<_W$.  This weak order
is the transitive closure of the relation
in which $w s_i < _W w$ if $w=(w_1,\ldots,w_n)$ has $w_i > w_{i+1}$, where $s_i=(i,i+1)$ is an adjacent transposition.
One can view this weak order as induced from the action of the {\it bubble-sorting operators} $\pi_1,\ldots,\pi_{n-1}$
on $\symm_n$
\begin{equation}
\label{bubble-sorting-operators}
w \bubble \pi_i=\begin{cases}
ws_i & \text{ if } w_i > w_{i+1},\\
w & \text{ if } w_i < w_{i+1},
\end{cases}
\end{equation}
which satisfy the relations of the {\it 0-Hecke monoid} of type $A_{n-1}$:
\begin{equation}
\label{zero-Hecke-relations}
\begin{array}{rcl}
\pi_i \pi_j &=& \pi_j \pi_i \text{ if } |j-i|\geq 2,\\
\pi_i \pi_{i+1} \pi_i &=& \pi_{i+1} \pi_i \pi_{i+1},\\
\pi_i^2&=&\pi_i. 
\end{array}
\end{equation}
Note that $\pi_i$ acts on right. 
This notational choice highlights the relationship between the application of $\pi_i$ and multiplication on the right by $s_i$.
One may then define the (right) weak order by $w \leq_W w'$ if and only if $w$ lies in the $0$-Hecke orbit of $w'$.

Section~\ref{s:zero-Hecke} extends this $0$-Hecke action from $\symm_n$ to $\MT_n$, by letting $T \bubble \pi_i$ replace
the $i^{th}$-row of the monotone triangle $T$ with the componentwise smallest row that still forms a monotone triangle with the remaining rows.
One can then extend the weak order $<_W$ from $\symm_n$ to $\MT_n$ by setting $T \leq T'$ whenever
$T$ lies in the $0$-Hecke orbit of $T'$.    For $n=3$, these actions of $H_3(0)$ on $\symm_3$ and $\MT_3$ look as follows,
illustrating the weak order posets $<_W$ on both:
$$
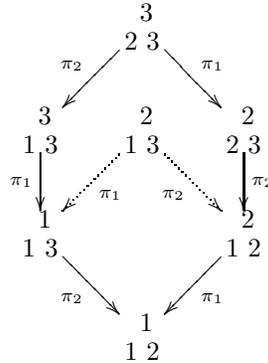
\xymatrix@R=.3in@C=.2in{ 
                                &*\txt<2pc>{ 3 \\2 3}\ar_{\pi_2}[dl]\ar^{\pi_1}[dr]& \\
*\txt<2pc>{ 3 \\1 3}\ar_{\pi_1}[d]&*\txt<2pc>{ 2 \\1 3}\ar@{.>}^{\pi_1}[dl]\ar@{.>}_{\pi_2}[dr] &*\txt<2pc>{ 2 \\2 3}\ar^{\pi_2}[d] \\
*\txt<2pc>{ 1 \\1 3}\ar_{\pi_2}[dr]& &*\txt<2pc>{ 2 \\1 2}\ar^{\pi_1}[dl] \\
 &*\txt<2pc>{ 1 \\1 2}&  
} 
$$

Section~\ref{s:weak-order-shelling} then uses this to prove our first main result.

\begin{theorem}
\label{thm:weak-order-gives-shellings}
Linear extensions of $<_W$ on $\MT_n$ give shelling orders on $\Phi_n$.
\end{theorem}

There is another sense in which the terminology {\it weak order} is appropriate.
Lascoux and Sch\"utzenberger ~\cite{lascoux1996treillis} showed that the 
componentwise order on $MT_n$ is a {\it distributive lattice}, one that turns out to be the {\it MacNeille completion} of the {\it (strong) Bruhat order} $<_B$ on $\symm_n$; 
we therefore refer to this componentwise order on $MT_n$ as its {\it (strong) Bruhat order} $<B$.   Depicted below is the the poset $(\MT_3,<_B)$, with the usual Bruhat order $(\symm_3,<_B)$ as a subposet, 
and dotted edges indicating the order relation to the unique element $T$ in $\MT_3 \setminus \symm_3$:
$$
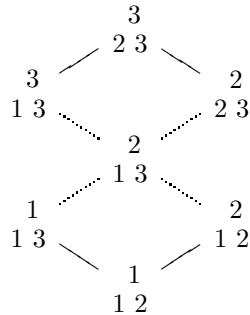
\xymatrix@R=.1in@C=.2in{ 
                                &*\txt<2pc>{ 3 \\2 3}\ar@{-}[dl]\ar@{-}[dr]& \\
*\txt<2pc>{ 3 \\1 3}\ar@{.}[dr]&                                                                                           &*\txt<2pc>{ 2 \\2 3}\ar@{.}[dl] \\
                                                  &*\txt<2pc>{ 2 \\1 3}\ar@{.}[dl]\ar@{.}[dr] & \\
*\txt<2pc>{ 1 \\1 3}\ar@{-}[dr]& &*\txt<2pc>{ 2 \\1 2}\ar@{-}[dl] \\
 &*\txt<2pc>{ 1 \\1 2}&  
}
$$
It turns out (see Remark~\ref{why-called-weak-remark}) that this Bruhat order $<_B$ on $\MT_n$ is stronger than the weak order $<_W$ defined above;  in particular,
any linear extension of the componentwise order gives rise to a shelling of $\Phi_n$.  

The weak order shellings provided by Theorem~\ref{thm:weak-order-gives-shellings} have another
tight analogy to the weak order shellings of the Boolean posets $(2^{[n]}, \subseteq)$, in that they contain as a special case certain
{\it EL-shellings}, a notion which we recall here.  
Given a poset $P$, with $C(P)=\{x \lessdot y: x,y \in P\}$ its set of {\it cover relations} ($x \lessdot y$ means $x <y$ but $\not\exists z$ with $x<z<y$),
an \emph{EL-labeling} of $P$ is a function $\lambda:C(P) \to \Lambda$ where $(\Lambda,<_{\Lambda})$ is any poset,
having these properties:
\begin{itemize}
\item for every interval $[x,y] \subset P$, there is a unique maximal chain $(x = x_0 \lessdot x_1 \lessdot \dots \lessdot x_k = y)$, that has {\it weakly rising} labels 
$$
\lambda(x_{0},x_{1}) \leq_{\Lambda} \lambda(x_{1},x_{2}) \leq_{\Lambda} \cdots  \leq_{\Lambda} \lambda(x_{k-1},x_{x_k})
$$
\item if $x \lessdot z < y$, with $z \neq x_1$, then $\lambda(x,x_1) <_{\Lambda} \lambda(x,z)$. 
\end{itemize}
For example, the Boolean algebras $(2^{[n]}, \subseteq)$ have a very simple EL-labeling.
It assigns a covering relation between subsets $I \subset J$ with $\#J=\#I+1$ the
unique integer $\lambda(I,J):=j$ such that $J=I \cup \{j\}$; here the labels come from the poset
$\Lambda=\{1,2,\ldots,n\}$ with the usual ordering on integers.
A poset is \emph{EL-shellable} or \emph{lexicographically shellable} if it admits an EL-labeling.
Bj\"orner \cite[Thm. 2.3]{bjorner1980shellable} showed that for a poset with an EL-labeling, one obtains a shelling order on
its maximal chains via any linear extension of the lexicographic extension of $\Lambda$ to sequences of edge labels.
In Section~\ref{s:EL}, we prove the following.

\begin{theorem}
\label{t:asm_el}
There is a partial order on $\Alt_n$ so that the edge-labeling $\lambda$ which assigns 
$\lambda(I \lessdot J) = \one_J - \one_I$ in $\Alt_n$ becomes an EL-labeling of $\Phi_n$.
Furthermore, any of the EL-shelling orders associated with this labeling will be a linear order on $\MT_n$ that extends the weak order $<_W$. 
\end{theorem}

The weak order shellings and EL-shellings in 
Theorems~\ref{thm:weak-order-gives-shellings}, \ref{t:asm_el} show that $\Phi_n$ is a {\it Cohen-Macaulay} poset, and allow one to combinatorially re-interpret
its {\it flag $f$-vector} $f(\Phi_n):=(f_J)_{J \subset [n-1]}$;  here 
$f_J$ is the number of chains in $\Phi_n$ 
that pass through the ranks in $J$.  
One can instead consider the {\it flag $h$-vector} $h(\Phi_n)=(f_J)_{J \subset [n-1]}$, 
defined by an inclusion-exclusion
relation:
$$
\begin{aligned}
f_J&=\sum_{ I \subseteq  J } h_I, \quad \text{ or equivalently, }\\
h_J&=\sum_{ J \subseteq I } (-1)^{\#J \setminus I } f_I.
\end{aligned}
$$
General shelling theory then implies this combinatorial interpretation for $h_J$:
$$
h_J(\Phi_n)=\# \{ T \in \MT_n: \Des(T)=J \}.
$$
Here one is led to define the {\it descent set} $\Des(T)$ for a monotone triangle $T$ as follows via the following generalization of the usual descent set $\Des(w)=\{i \in [n-1]: w_i > w_{i+1} ,\text{ that is, } w \bubble \pi_i \neq w\}$ for permutations $w$ in $\symm_n$:
$$
\Des(T):=\{i \in [n-1]: T \bubble \pi_i \neq T\}.
$$
Section~\ref{s:descents} discusses this 
descent set $\Des(T)$, and collects some data on its distribution over $\MT_n$.

There is a further way in which this notion of a descent set for monotone triangles extends a pleasant property of descents for permutations.  Recall that Malvenuto and Reutenauer \cite{malvenuto1995duality} defined a graded Hopf algebra, sometimes denoted $\FQSym=\bigoplus_{n \geq 0} \FQSym_n$, where  $\FQSym_n$ has $\Z$-basis elements $\textbf{w}$ indexed by permutations $w$ in $\symm_n$.  The ring structure is determined by a {\it shuffle product} for $u,v$ in $\symm_n,\symm_m$
defined as
$$
\textbf{u} \textbf{v} = \sum_{w \in u \shuffle v[n]} \textbf{w}
$$
in which the sum runs over all shuffles $w$ of $u=(u_1,\ldots,u_n)$, and
$v[n]=(v_1+n,\ldots,v_m+n)$.  This shuffle product was introduced in such a 
way as to make a ring (and Hopf algebra) morphism into
the  {\it quasisymmetric functions} $\QSym$, defined by
\begin{equation}
\label{FQsym-to-QSym-map} 
\begin{array}{rcl}
\mathcal{\FQSym} &\longrightarrow& \QSym\\
\textbf{w} &\longmapsto& L_{\alpha(\Des(w))}.
\end{array}
\end{equation}
Here $L_{\alpha}$ denotes Gessel's {\it fundamental quasisymmetric function} 
 associated to a composition $\alpha$, and  $\alpha(\Des(w))$ is the composition whose partial sums give
 the elements of $\Des(w)$; see \cite[\S7.19]{Stanley-EC2} and Section~\ref{s:qsym} below.  Recently, Cheballah, Giraudo and Maurice embedded $\FQSym$ inside a larger 
graded Hopf algebra $\mathcal{ASM}$ whose $n^{th}$-graded component has a 
basis $\{\textbf{A}\}$ indexed by $A$ in $\ASM_n$~\cite{cheballah2015hopf},
and whose product and coproduct extend that of $\FQSym$.  Section~\ref{s:qsym} proves the following.

\begin{theorem}
\label{t:alg_map}
The map $\FQSym \rightarrow \QSym$ in \eqref{FQsym-to-QSym-map} 
extends to an algebra (but not a coalgebra) morphism 
$$
\begin{array}{rcl}
\mathcal{ASM} &\longrightarrow& \QSym\\
\textbf{A} &\longmapsto& L_{\alpha(\Des(A))}
\end{array}
$$ 
where $\Des(A)=\Des(T(A))$ for an alternating sign matrix $A$ is the descent set of its monotone triangle $T(A)$.
\end{theorem}

Sections~\ref{pathology-section}
concludes by comparing poset properties of the weak order on $\MT_n$ with
analogous properties for the weak order on $\symm_n$, including a conjecture for
the homotopy type of open intervals in $(\MT_n,<_W)$.

\subsection*{Acknowledgements} 
The second author was partially supported by NSF grant DMS-1601961.
The authors thank Ilse Fischer, Darij Grinberg, John Harding, Brendon Rhoades, John Stembridge and Jessica Striker for helpful discussions, 
and thank Brendan Pawlowski for sharing his code to compute MacNeille completion of posets.
In addition, we are grateful to Roger Behrend for detailed feedback on an earlier draft leading to numerous improvements, including his illuminating example.
This work began during the Fall 2017 MSRI semester in
Geometric and Topological Combinatorics.

\section{Interlacing, monotone trapezoids, and the order $\Phi_n$}
\label{s:interlacing}

The goal here is to relate 
Terwilliger's order $\Phi_n$ 
with the notions of interlacing and monotone trapezoids.

\begin{definition} \rm \ \\
Start with the {\it componentwise} order $<_{\comp}$ on subsets $I, I' \subset [n]$ of the same cardinality 
$k$ for $0 \leq k \leq n$,
$$
\begin{aligned}
I&=\{i_1 < i_2 < \cdots < i_k\},\\
I&'=\{i'_1 < i'_2 < \cdots <i'_k\},
\end{aligned}
$$
defined by setting $I \leq_{\comp} I'$ if $i_m \leq i'_m$ for $m=1,2,\ldots,k$.

For $J=\{j_1 < \cdots < j_\ell\} \subset [n]$ with $\#J= \ell \geq k=\#I$,
say that {\it $J$ interlaces $I$}, written $I \leq_{\interlace} J$,
if 
$$
\{j_1,j_2,\ldots,j_k\} \leq_{\comp} 
              I \leq_{\comp} \{j_{\ell-k+1},j_{\ell-k+2},\ldots,j_{\ell-1},j_\ell\}.
$$
\end{definition}

Note that when $\#J=k+1=\#I+1$, this condition $I \leq_{\interlace} J$ is the 
usual definition of $J$ interlacing $I$, as given in
\eqref{classical-interlacing-definition} earlier.
One then has the following proposition which is easily checked (or see \cite[\S 3]{terwilliger2017poset}).

\begin{proposition}
\label{interlace-covering-is-alternating-difference-prop}
If $\#J=\#I+1$, then $I \leq_{\interlace} J$ if and only if
$\one_J-\one_I$ lies in $\Alt_n \qed$.
\end{proposition}

One can also readily check that $\leq_{\interlace}$ is a partial order,
that is, $I \leq_{\interlace} J \leq_{\interlace} K$ implies $I \leq_{\interlace} K$.
This partial  order $\leq_{\interlace}$ is closely related to {\it monotone trapezoids} and 
Terwilliger's order $\Phi_n$, as we now explain.

\begin{definition} \rm \ \\
An {\it $(I,J)$-monotone trapezoid} 
is a sequence of subsets $T=(I_k,I_{k+1},\ldots,I_{\ell-1},I_\ell)$ of $\{1,2,\ldots\}$ with
\begin{itemize}
\item $I_k=I, I_\ell=J$,
\item $\#I_m=m$, and
\item $I_m \leq_{\interlace} I_{m+1}$ for $k\leq m< \ell$.  
\end{itemize}
In other words, an $(I,J)$-monotone trapezoid is a saturated chain in $\leq_{\interlace}$ from $I$ to $J$.
When $(I,J)=(\varnothing, [n])$, one calls $T$ a {\it monotone triangle} of
size $n$.  
\end{definition}

\begin{proposition}
\label{interlacing-equivalences-prop}
The following are equivalent for subsets $I, J \subseteq [n]$:
\begin{enumerate}
\item[(a)] There exists at least one $(I,J)$-monotone trapezoid.
\item[(b)] $I \leq_{\Phi_n} J$. 
\item[(c)] $I \leq_{\interlace} J$.
\end{enumerate}
\end{proposition}

In proving this proposition, and in the sequel, the following
construction will be useful.  

\begin{definition} \rm \ \\
\label{def-of-Hmin}
For $I \leq_{\interlace} J$ with $\#I=k$
and $\#J \geq k+2$, define 
$H_{\min}(I,J):=\{h_1,h_2, \ldots, h_{k+1}\}$ by the rule 
\begin{equation}
\label{Hmin-formula}
h_m:=\max(i_{m-1},j_m),
\end{equation}
and convention $i_p:=0$ for $p = 0$. Thus when $k=0$, 
so that $I=\varnothing$, then $H_{\min}(\varnothing,J)=\{j_1\}$.
\end{definition}

\begin{lemma}
\label{minimal-H-lemma}
The set $H_{\min}(I,J)$ has these properties: 
\begin{itemize}
\item[(i)]
It is a $(k+1)$-subset, that is, $h_1<\cdots<h_{k+1}$.
\item[(ii)]
It lies in the family
$\{H \in \binom{[n]}{k+1}: I \leq_{\interlace} H \leq_{\interlace} J\}$.
\item[(iii)]
Every $H'$ in this family has $H_{\min}(I,J) \leq_{\comp} H'$.
\end{itemize}
\end{lemma}
\begin{proof}
\vskip.1in
\noindent
{\sf Assertion (i).}
The definition of $H_{\min}(I,J)$ implies $h_m < h_{m+1}$ since
$$
h_m=\max(i_{m-1},j_m) 
\leq \max(i_m-1,j_{m+1}-1)
=    \max(i_m,j_{m+1})-1
=h_{m+1}-1.
$$

\vskip.1in
\noindent
{\sf Assertion (ii).}
We must show two $\leq_{\interlace}$-inequalities, or
equivalently, four $\leq_{\comp}$-inequalities.
\begin{itemize}
\item  Two of the four
come from $i_{m-1}, j_m \leq \max(i_{m-1},j_m)=h_m$ for 
$m=1,2,\ldots,k+1$, which shows both that $I \leq_{\comp} \{h_2,\ldots,h_{k+1}\}$
and also that $\{j_1,\ldots,j_{k+1}\} \leq_{\comp} H_{\min}(I,J)$.
\item
The inequality
$\{h_1,\ldots,h_k\} \leq_{\comp} I$ comes 
from 
$$
h_m=\max(i_{m-1},j_m) \leq \max(i_m,j_m)=i_m
$$
which uses $i_{m-1}<i_m$ and the fact that $\{j_1,\dots,j_m\} \leq_{\comp} I$ since $I \leq_{\interlace} J$.

\item The last inequality 
$H_{\min}(I,J) \leq \{j_{\ell-k},j_{\ell-k+1},\ldots,j_{\ell-1},j_\ell\}$
comes from 
$$
h_m=\max(i_{m-1},j_m) \leq j_{\ell-k+(m-1)}
$$
which uses
$j_m < j_{\ell-k+(m-1)}$ (as $\ell-k \geq 2$) and
$i_{m-1} \leq j_{\ell-k+(m-1)}$ (as $I \leq_{\interlace} J$).
\end{itemize}

\vskip.1in
\noindent
{\sf Assertion (iii).}
Any such $H'=\{h'_1<\cdots<h'_{k+1}\}$
has $I \leq_{\interlace} H' \leq_{\interlace} J$,
implying for $1 \leq m \leq k+1$ that
\begin{itemize}
\item $h'_m \geq i_{m-1}$, coming from $I \leq_{\comp} \{h'_2,h_3',\ldots,h'_{k+1}\}$,
\item $h'_m \geq j_m$, coming from $\{j_1,\ldots,j_m\} \leq_\comp H'$.
\end{itemize}
Thus $h'_m \geq \max(i_{m-1},j_m)=h_m$, 
that is, $H_{\min}(I,J)  \leq_{\comp} H'$, as desired.
\end{proof}

With the construction $H_{\min}(I,J)$ and its properties in hand, one can now prove 
Proposition~\ref{interlacing-equivalences-prop}.

\begin{proof}[Proof of Proposition~\ref{interlacing-equivalences-prop}.]
Note (a) $\Leftrightarrow$ (b) via 
Proposition~\ref{interlace-covering-is-alternating-difference-prop}
and definition of $\Phi_n$.
Then (a) $\Rightarrow$ (c) from the transitivity of $\leq_{\interlace}$,
while (c) $\Rightarrow$ (a) follows by induction on $\#J - \#I$ via
Lemma~\ref{minimal-H-lemma}.
\end{proof}

\begin{remark} \rm \ \\
\label{H-max-remark}
It is worth pointing out an involutive poset symmetry in $\Phi_n$, coming from the action of the
longest permutation $w_0=(n,n-1,\ldots,2,1)$ in $\symm_n$.  This permuation $w_0$ acts on subsets as follows:
$$
I=\{i_1 < i_2 < \cdots < i_k\} \quad \overset{w_0}{\longmapsto} \quad w_0(I):=\{n+1-i_k < \cdots < n+1-i_2 < n+1-i_1\}.
$$
Since $i \leq j$ if and only if $n+1-i \geq n+1-j$, this action of $w_0$ preserves the interlacing inequalities
\eqref{classical-interlacing-definition} that define the covering relations $I \lessdot_{\Phi_n} J$.  Thus
it is an involutive automorphism of the poset $\Phi_n$, and therefore also gives an involution on monotone triangles 
$$
T=(T_0,T_1,\ldots,T_n) \quad \overset{w_0}{\longmapsto} \quad  w_0(T):=(w_0(T_0),w_0(T_1),\ldots,w_0(T_n)).
$$
Passing through the bijection $\ASM_n \leftrightarrow \MT_n$, the corresponding involution $w_0$ acting
on a matrix $A=(a_{ij})$ in $\ASM_n$ simply reflects it through a vertical axis: $w_0(A):=(a_{i,n+1-j})$.

Due to this $w_0$-symmetry, for $I <_\interlace J$ with $\#J - \#I \geq 2$,
instead of defining the set $H_{\min}(I,J)$ as in Definition~\ref{def-of-Hmin}, we could have defined a set 
$H_{\max}(I,J)=\{h_1' < h_2' < \cdots < h'_{k+1}\}$ via two equivalent formulas:
\begin{equation}
\label{H-max-definition}
\begin{aligned}
h_m'&=\min(i_m, j_{m-1+\ell-k}) \text{ for }m=1,2,\ldots,k+1, \text{ with convention }i_{k+1}:=\infty,\text{ or }\\
H_{\max}(I,J) &= w_0( H_{\min}( w_0(I), w_0(J) )).
\end{aligned}
\end{equation}
One would then have the corresponding properties as in Lemma~\ref{minimal-H-lemma}, namely that
$H_{\max}(I,J)$ is actually a $(k+1)$-subset, that it lies between $I$ and $J$ in the order $<_{\interlace}$,
and that it is the componentwise {\it maximum} among all such $(k+1)$-subsets between $I$ and $J$.  We simply
chose  here to use $H_{\min}(I,J)$, not $H_{\max}(I,J)$.
\end{remark}

The key property that we will need for shellability of $\Phi_n$ is that, for any pair $I \leq_{\interlace} J$, 
there is a componentwise smallest $(I,J)$-monotone trapezoid,
and that it can be characterized {\it locally}.

\begin{lemma}
\label{minimal-trapezoid-lemma}
Fixing $I \leq_{\interlace} J$, the 
following are equivalent for an $(I,J)$-monotone trapezoid
$$
T:=((I=)I_k,I_{k+1},\ldots,I_{\ell-1},I_\ell(=J)):
$$
\begin{enumerate}
\item[(a)]
  $I_m=H_{\min}(I_{m-1},J)$ for $m=k+1,k+2,\ldots,\ell-1$.
\item[(b)]
 $I_m=H_{\min}(I_{m-1},I_{m+1})$ for $m=k+1,k+2,\ldots,\ell-1$.
\item[(c)]
The elements of 
  $I_m=\{h^{(m)}_1 < h^{(m)}_2 \cdots < h^{(m)}_m\}$ 
are           $h^{(m)}_p=\max(j_p,i_{p+k-m})$ with $i_q = 0$ for $q \leq 0$.
\item[(d)]
  $T$ is the componentwise smallest among all $(I,J)$-monotone trapezoids.
\end{enumerate}
\end{lemma}
\begin{proof}
First check that if $T$ satisfies (a), then its entries have the formula
from (c), using induction on $m$.
The base case $m=k+1$ comes from the definition of $H_{\min}(I_k,J)$.
The inductive step is this calculation:
$$
h_p^{(m)}=\max(j_p,h^{(m-1)}_{p-1}) 
= \max(j_p, \max(j_p, i_{p-1+k-(m-1)}) )
= \max(j_p, i_{p+k-m}).
$$

Next check that if $T$ satisfies (b), then its entries obey the formula from (c), 
this time using induction on $\#J-\#I=\ell-k$.
Assume that (b) holds for the trapezoid $T$, so \[
I_m=\{h^{(m)}_1 < h^{(m)}_2 \cdots < h^{(m)}_m\}=H_{\min}(I^{(m-1)},I^{(m+1)}).
\]
This means that 
\begin{equation}
\label{squeezed-H-entry}
h^{(m)}_p = \max( h_p^{(m+1)}, h_{p-1}^{(m-1)} ).
\end{equation}
By restriction, condition (b) also holds for the smaller trapezoid
$( I_m,I_{m+1},\ldots,I_{\ell-1},I_\ell=J)$, and hence by induction,
one has $h_p^{(m+1)}=\max(j_p,h^{(m)}_{p-1})$.
Similarly, by restriction, condition (b) also holds for the smaller trapezoid
$( I=I_k,I_{k+1}, \ldots,I_{m-1},I_m)$, and hence by induction,
one has $h_{p-1}^{(m-1)}=\max(h_{p-1}^{(m)},i_{p-1+k-(m-1)})$.
Plugging these last two expressions into \eqref{squeezed-H-entry},
one concludes that 
$$
\begin{aligned}
h^{(m)}_p &= \max( \max(j_p,h^{(m)}_{p-1}), \max(h_{p-1}^{(m)},i_{p-1+k-(m-1)}) )\\
&=\max(j_p,h_{p-1}^{(m)},i_{p+k-m}))
=\max(j_p,i_{p+k-m})
\end{aligned}
$$
since $h_{p-1}^{(m)} < h^{(m)}_p$.  This last expression is the one from (c),
as desired.

Thus since (a) does define a monotone trapezoid
having $I,J$ as its bottom, top rows, then $T$ satisfiying 
(b) or (c) is equivalent to $T$ being the one defined by (a).

To see (c) $\Leftrightarrow$ (d), 
let $T'=((I=)I'_k,I'_{k+1},\ldots,I'_{\ell-1},I'_\ell(=J))$ be an $(I,J)$-monotone trapezoid, with $I_m'=\{i_1'<\ldots < i_m'\}$.  Then $i'_{p} \geq \max(j_p,i_{p+k-m})$ 
by the inequalities defining monotone trapezoids.
Since the sets defined using (c) form an $(I,J)$-monotone trapezoid, we see they must form the minimal $(I,J)$-monotone trapezoid and vice versa.
%
%
%
\end{proof}

\begin{remark} \rm \ \\
\label{r:max}
It should not be surprising that there exists a componentwise
smallest $(I,J)$-monotone trapezoid, as in Lemma~\ref{minimal-trapezoid-lemma},
since Lascoux and Sch\"utzenberger \cite[\S 5]{lascoux1996treillis} showed that 
the componentwise
order on $\MT_n$ has meet and join operations given by componentwise minimum and
maximum.  Similarly, there is a componentwise largest such $(I,J)$-monotone trapezoid,
having similar properties, which can be built in a analogous fashion by
iterating the $H_{\max}(I,J)$ construction from Remark~\ref{H-max-remark}.

\end{remark}

\section{Action of $\HHH_n(0)$ and the weak order}
\label{s:zero-Hecke}

Recall from the Introduction \eqref{zero-Hecke-relations} that the {\it $0$-Hecke monoid}
$\HHH_n(0)$ for the symmetric group $\symm_n$ (or type $A_{n-1}$) is the monoid
with $n-1$ generators $\pi_1,\pi_2,\ldots,\pi_{n-1}$
subject to the usual {\it braid relations}
\begin{equation}
\label{braid-relations}
\begin{array}{rcll}
\pi_i \pi_j &=& \pi_j \pi_i &\text{ for }|i-j| \geq 2,\\
\pi_i \pi_{i+1} \pi_i &=& \pi_{i+1} \pi_i \pi_{i+1}  &\text{ for }i=1,2,\ldots,n-2,
\end{array}
\end{equation}
together with the {\it quadratic relations} 
\begin{equation}
\label{0-Hecke-quadratic-relation}
\pi_i^2=\pi_i\text{ for }i=1,2,\ldots,n-1.
\end{equation}
See Norton \cite{Norton} for background on $\HHH_n(0)$ and the associated
monoid algebra, called  a {\it $0$-Hecke algebra}.

\begin{definition} \rm \ \\
\label{d:pi}
Define maps $\pi_i: \MT_n \longrightarrow \MT_n$ for
$i=1,2,\ldots,n-1$ sending
$T \mapsto T \bubble \pi_i$, where $T \bubble \pi_i$ is
obtained from $T$ by replacing
its $i^{th}$ row $T_i$ with $H_{\min}(T_{i-1},T_{i+1})$.
\end{definition}

\begin{proposition}
\label{p:0-hecke}
The operators $\pi_i$ on $\MT_n$ satisfy the braid and quadratic relations \eqref{braid-relations}, \eqref{0-Hecke-quadratic-relation},
and hence define an action of $\HHH_n(0)$ on $\MT_n$.
\end{proposition}
\begin{proof}
The relations $\pi_i^2=\pi_i$ and $\pi_i \pi_j = \pi_j 
\pi_i$ for $|i-j|\geq 2$ should
be clear; only $\pi_i \pi_{i+1} \pi_i = \pi_{i+1} \pi_i \pi_{i+1}$ 
requires verification.
We can check this locally in rows $i-1,i,i+1,i+2$ of a monotone triangle $T$, by tracking two generic entries in rows $i, i+1$ shown in bold below.
Here, we are using concatenation of sets of entries to abbreviate their
maximum:

$$
\xymatrix@R=.01in@C=.01in{
 &a& &b& \\
c& &d& &\mathbf{e}\\
 &f& &\mathbf{g}& \\
h& &i& &
}
\quad\begin{matrix} \\ \\ \overset{\pi_i}{\longmapsto}\\  \end{matrix}\quad
\xymatrix@R=.01in@C=.01in{
 &a& &b& \\
?& &af& &\mathbf{bg}\\
 &f& &\mathbf{g}& \\
h& &i& &
}
\quad\begin{matrix} \\ \\ \overset{\pi_{i+1}}{\longmapsto}\\  \end{matrix}\quad
\xymatrix@R=.01in@C=.01in{
 &a& &b& \\
?& &af& &\mathbf{bg}\\
 &?& &\mathbf{afi}& \\
h& &i& &
}
\quad\begin{matrix} \\ \\ \overset{\pi_{i}}{\longmapsto}\\  \end{matrix}\quad
\xymatrix@R=.01in@C=.01in{
 &a& &b& \\
?& &?& &\mathbf{abfi}\\
 &?& &\mathbf{afi}& \\
h& &i& &
}
$$
$$
\xymatrix@R=.01in@C=.01in{
 &a& &b& \\
c& &d& &\mathbf{e}\\
 &f& &\mathbf{g}& \\
h& &i& &
}
\quad\begin{matrix} \\ \\ \overset{\pi_{i+1}}{\longmapsto}\\  \end{matrix}\quad
\xymatrix@R=.01in@C=.01in{
 &a& &b& \\
c& &d& &\mathbf{e}\\
 &ch& &\mathbf{di}& \\
h& &i& &
}
\quad\begin{matrix} \\ \\ \overset{\pi_{i}}{\longmapsto}\\  \end{matrix}\quad
\xymatrix@R=.01in@C=.01in{
 &a& &b& \\
?& &ach& &\mathbf{bdi}\\
 &ch& &\mathbf{di}& \\
h& &i& &
}
\quad\begin{matrix} \\ \\ \overset{\pi_{i+1}}{\longmapsto}\\  \end{matrix}\quad
\xymatrix@R=.01in@C=.01in{
 &a& &b& \\
?& &ach& &\mathbf{bdi}\\
 &?& &\mathbf{achi}& \\
h& &i& &
}
$$
Thus it only remains to check these equalities
\begin{eqnarray}
\label{first-putative-equality}
\max(a,b,f,i)&\overset{?}{=}&\max(b,d,i),\\
\label{second-putative-equality}
\max(a,f,i)&\overset{?}{=}&\max(a,c,h,i),
\end{eqnarray}
which both follow, since
\begin{itemize}
\item
 $a \leq d \leq b$ and $f \leq i$ implies that the two sides in  \eqref{first-putative-equality} are both equal
to $\max(b,i)$,
\item
$c,h\leq f \leq i$ implies that the two sides in \eqref{second-putative-equality} are both equal 
to $\max(a,i)$. $\qedhere$
\end{itemize}
\end{proof}

Once one knows that the operators $\pi_i$ satisfy the braid relations, one can define operators $\pi_w$ for every 
permutation $w$ in $\symm_n$ as follows:  pick any factorization 
$w=s_{i_1} s_{i_2} \cdots s_{i_\ell}$ for $w$ that is shortest possible (i.e., {\it reduced}) 
as a product of the adjacent transpositions $\{s_1,s_2,\ldots,s_{n-1}\}=:S$,
and then let 
$$
\pi_w:=\pi_{i_1} \pi_{i_2} \dots \pi_{i_\ell}.
$$
As a consequence of satisfying the relations of $\HHH_n(0)$, one could equivalently define $\pi_w$ recursively as follows:
\begin{equation}
\label{recursive-pi-definitions}
\pi_w \pi_i  := \begin{cases}
\pi_{w s_i} & \text{ if }w(i) < w(i+1), \text{ that is, if }i \not\in \Des(w),\\
\pi_{w} & \text{ if }w(i) > w(i+1), \text{ that is,  if }i \in \Des(w),
\end{cases}
\end{equation}
starting with the initial condition $\pi_e:=1$.

\begin{remark} \rm \ \\
\label{r:extend}
It is worth noting in the case where $T$ has $T_i \subset T_{i+1}$ for all $i$, so that 
\[
T=T(w):=(\varnothing, \{w_1\},\{w_1,w_2\},\ldots,\{w_1,w_2,\ldots,w_{n-1}\},[n])
\]
for some permutation $w=(w_1,w_2,\ldots,w_n)$ in $\symm_n$,
then one has
\[
T(w) \bubble \pi_i =  \begin{cases}
T(w) & \text{ if } w_i < w_{i+1}, \text{ that is, if }i \not\in \Des(w),\\
T(w s_i) & \text{ if } w_i > w_{i+1},\text{ that is, if }i \in \Des(w).\\
\end{cases}
\] 
Here $s_i=(i,i+1)$ is the adjacent transposition,
so that 
\[
ws_i=(w_1, w_2,\ldots,w_{i-1},w_{i+1},w_i,w_{i+2},\ldots,w_{n-1},w_n).
\]
Thus the action of $\HHH_n(0)$ on $MT_n$ extends its action 
on $\symm_n$ via {\it (bubble-)sorting operators} as mentioned in the Introduction. 
We let $w \bubble \pi_i$ denote the permutation corresponding to $T(w) \bubble \pi_i$,
so that $T(w) \bubble \pi_i=T(w \bubble \pi_i)$.
\end{remark}

\begin{definition} \rm \ \\
\label{d:weak_order}
Extend the {\it weak order} $<_W$ on the symmetric group $\symm_n$ to 
a weak order $<_W$ on monotone triangles $MT_n$ as 
the transitive closure of the relations $T \bubble \pi_i \leq T$
where $i$ is any index in the range $1,2,\ldots,n-1$.  
Equivalently, $T \leq_W T'$ means
that $T$ lies in the $\HHH_n(0)$-orbit of $T'$.
\end{definition}

\begin{remark} \rm \ \\
\label{why-called-weak-remark}
The name {\it weak order} is appropriate here,
since $(\MT_n,<_W)$ is indeed weaker than the {\it componentwise} order $(\MT_n,<_B)$,
and we view the latter as the appropriate extension of (strong) Bruhat order on $\symm_n$ to a strong Bruhat order on 
$\MT_n$, via MacNeille completion.  To see that $(\MT_n,<_W)$ is weaker than the componentwise order, note that
it is the transitive closure of the relations $T \bubble \pi_i \leq_W T$,  where $T \bubble \pi_i$ is obtained from $T$
by replacing the $i^{th}$ row of $T$ with  $H_{\min}(T_{i-1},T_{i+1})$,
the latter being componentwise smaller by Lemma~\ref{minimal-H-lemma}.
\end{remark}

\section{Proof of Theorem~\ref{thm:weak-order-gives-shellings}}
\label{s:weak-order-shelling}

Recall the statement of the theorem.
\vskip.1in
\noindent
{\bf Theorem~\ref{thm:weak-order-gives-shellings}.}
{\it 
Any linear extension $T^{(1)},T^{(2)},\cdots,T^{(N)}$ of $<_W$ on $\MT_n$ gives a shelling order on $\Phi_n$.
}
\vskip.1in

Before proving the theorem, we note in the next proposition a useful reinterpretation of Lemma~\ref{minimal-trapezoid-lemma}, generalizing the definition of the $T \bubble \pi_i$ on monotone triangles.  Given any subset $J \subseteq S:=\{s_1,\ldots,s_{n-1}\}$, recall there is a unique longest permutation $w_0(J)$ in the {\it (Young or parabolic) subgroup}  $\langle J \rangle$ of $\symm_n$ generated by $J$.  This $w_0(J)$ is an involution, characterized within $\langle J \rangle$ by the property that 
\begin{equation}
\label{descents-of-w0}
J=\Des(w_0(J)) (=\Des(w_0(J)^{-1}))
\end{equation}
(here we identify $J = \{s_{j_1},\dots,s_{j_k}\}$ with $\{j_1,\dots,j_k\}$). For example, if $n=9$ and $J=\{s_1,s_2, s_4,s_5,s_6,s_8\} \subset \{s_1,s_2,\ldots,s_8\}=S$,
then the parabolic subgroup $\langle J \rangle$ inside $\symm_{9}$ is the subgroup isomorphic to
$\symm_3 \times \symm_4 \times \symm_2$ that stabilizes the blocks of the
partition $\{1,2,3\},\{4,5,6,7\},\{8,9\}$.  Its longest permutation is $w_0(J)=(3,2,1,\,7,6,5,4,\,9,8)$.

\begin{proposition}
\label{Hecke-action-versus-minimum-triangles}
Given any monotone triangle $T$ and $J \subseteq S$, then $T \bubble \pi_{w_0(J)}$ is the unique componentwise
smallest monotone triangle $T^{\min}$ having the same rows $T_m$ as $T$ for all $s_m \not\in J$.
\end{proposition}
\begin{proof}
Lemma~\ref{minimal-trapezoid-lemma}(b) shows that this componentwise smallest triangle $T^{\min}$ is uniquely characterized by 
$$
T^{\min}_m=\begin{cases}
T_m& \text{ for }s_m \not\in J,\\
H_{\min}(T^{\min}_{m-1},T^{\min}_{m+1}) &\text{ for }s_m \in J,
\end{cases}
$$
On the other hand, we claim that the triangle $T'= T \bubble \pi_{w_0(J)}$ has these same properties:
\begin{itemize}
\item $T'=T \bubble \pi_{w_0(J)}$ shares the same rows $T'_m=T_m$ for $s_m \not\in J$ since $w_0(J)$ lies in $\langle J \rangle$. 
\item For any $s_m \in J$ one has 
$
T' \bubble \pi_{m}
=T \bubble \pi_{w_0(J)} \pi_{m}  
= T \bubble \pi_{w_0(J)}
=T'
$
combining \eqref{recursive-pi-definitions} with the fact that
$s_m$  lies in $J=\Des(w_0(J))$ by \eqref{descents-of-w0}.
This means that $T'_m=H_{\min}(T'_{m-1},T'_{m+1})$. $\qedhere$
\end{itemize}
\end{proof}

\begin{proof}[Proof of Theorem~\ref{thm:weak-order-gives-shellings}.]
Thinking of each monotone triangle $T^{(i)}$ as corresponding to a facet, we identify it
with its subset of $n+1$ vertices, namely
$$
T^{(i)}=\{\varnothing=T^{(i)}_0,T^{(i)}_1,\ldots,T^{(i)}_{n-1},T^{(i)}_n=[n]\}.
$$
Shellability, as defined in the Introduction, requires that for each pair $i,j$ with $1 \leq i<j \leq N$, 
we must exhibit some $k<j$ satisfying $\# T^{(k)} \cap T^{(j)} = n$ (including $\varnothing$ and $[n]$)
and $T^{(i)} \cap T^{(j)} \subseteq T^{(k)} \cap T^{(j)}$.

Given $i < j$, let $J:=\{m: T^{(i)}_m  \neq T^{(j)}_m\}$.  We claim
that $ T^{(j)} \bubble \pi_m \neq T^{(j)}$ for at least one $m$ in $J$, otherwise
Proposition~\ref{Hecke-action-versus-minimum-triangles} implies the two equalities here
$$
T^{(j)} =  T^{(j)} \bubble \pi_{w_0(J)}
=  T^{(i)} \bubble \pi_{w_0(J)}
\leq_W T^{(i)},
$$
but then the inequality $T^{(j)} \leq_W T^{(i)}$  would contradict $i < j$. 

Given such an $m$, one checks that the index $k$ defined by $ T^{(j)} \bubble \pi_m=T^{(k)}$ does the job:
\begin{itemize}
\item $T^{(k)}= T^{(j)} \bubble \pi_m <_W T^{(j)}$ implies that $k < j$.
\item $\# \left(T^{(k)} \cap T^{(j)}\right) = \# \left(  T^{(j)} \bubble \pi_m \right) \cap T^{(j)} = n-1$, since $ T^{(j)} \bubble \pi_m \neq T^{(j)}$.
\item $T^{(i)} \cap T^{(j)} \subseteq T^{(k)} \cap T^{(j)}$ 
because $s_m$ lies in $J$. $\qedhere$
\end{itemize}
\end{proof}

We close this section with two remarks about the above shelling.

\begin{remark} \rm \ \\
Since the $\pi_i$ operators on $\MT_n$ restrict to the usual bubble-sorting operators
on the symmetric group $\symm_n$ embedded inside $\MT_n$ via $w \mapsto T(w)$, one
finds that the subposet $(\symm_n, <_W)$ is actually an {\it order ideal} inside
$(\MT_n,<_W)$; it is even the {\it principal order ideal} below 
$T(w_0)$ where $w_0=(n,n-1,\ldots,2,1)$.

As a consequence, it is possible to pick a linear extension of $<_W$ on $\MT_n$ which
contains all of the elements of the order ideal $\symm_n$ as an initial segment.  This
then gives a shelling order on the facets of $\Delta(\Phi_n \setminus \{ \hat{0},\hat{1} \})$
which shells the Coxeter complex $\Delta(2^{[n]} \setminus \{ \hat{0},\hat{1} \})$ {\it first},
before continuing on to shell the remaining facets of $\Delta(\Phi_n \setminus \{ \hat{0},\hat{1} \})$ that do not correspond to permutations.
\end{remark}

\begin{remark} \rm \ \\
Shellability implies that the $(n-2)$-dimensional simplicial complex
$\Delta(\Phi_n \setminus \{ \hat{0},\hat{1} \})$ has the 
homotopy type of a bouquet of $(n-2)$-spheres.  The Coxeter complex 
$\Delta(2^{[n]} \setminus \{ \hat{0},\hat{1} \})$ inside it is 
homeomorphic to a single $(n-2)$-sphere, and this sphere has well-known
easy embeddings into $\R^{n-1}$.  For example, it is isomorphic
to the {\it barycentric subdivision}
of the boundary of a simplex with vertex set $\{1,2,\ldots,n\}$.
Alternatively one can embed it within the hyperplane $x_1+\cdots+x_n=0$ inside
$\R^n$ by extending piecewise-linearly the map that sends its vertices 
to the $\symm_n$-images of the fundamental dominant weights of type $A_{n-1}$:  
the vertex indexed by a subset $I$ with $\varnothing \subsetneq I \subseteq [n]$ is sent to
the vector $\sum_{i \in I} e_i - \frac{\#I}{n}(e_1+\cdots+e_n)$ where $e_i$ is the $i^{th}$ standard basis vector of $\R^n$.

After looking at the picture  \eqref{simplicial-complexes-for-n=3} 
of $\Delta(\Phi_3 \setminus \{ \hat{0},\hat{1} \})$, which embeds it in $\R^2$, one might
wonder whether $\Delta(\Phi_n \setminus \{ \hat{0},\hat{1} \})$ embeds in some
simple way into $\R^{n-1}$.  We are doubtful.  For example, when $n=4$,
one can check that if one takes either of the two vertex coordinates for embedding
$\Delta(2^{[4]} \setminus \{ \hat{0},\hat{1} \})$ into $\R^3$ as described
in the previous paragaph, when one extends this piecewise-linearly over
the extra simplices in $\Delta(\Phi_4 \setminus \{ \hat{0},\hat{1} \})$,
it leads to self-intersections, and not an embedding.
\end{remark}

\section{EL-labeling and proof of Theorem~\ref{t:asm_el}}
\label{s:EL}

Recall the statement of the theorem.
\vskip.1in
\noindent
{\bf Theorem~\ref{t:asm_el}.}
{\it
There is a partial order on $\Alt_n$ so that the edge-labeling $\lambda$ which assigns 
$\lambda(I \lessdot J) = \one_J - \one_I$ in $\Alt_n$ becomes an EL-labeling of $\Phi_n$.
Furthermore, any of the EL-shelling orders associated with this EL-labeling is a linear order on $\MT_n$ which extends the weak order $<_W$. 
}
\vskip.1in

\noindent
We will define the partial order 
on $\Alt_n$ via its identification with a Boolean algebra $2^{[n-1]}$.
Note that a vector $v$ in $\{0,\pm 1\}^n$ lies in $\Alt_n$ exactly when
each of its {\it tail sums} 
$v \cdot \one_{[i,n]}= v_i+v_{i+1}+ \cdots+ v_n$ lies in $\{0,+1\}$, with
$\sum_{i=1}^n v_i=+1$.  The following proposition is straightforward to verify.

\begin{proposition}
One has mutually-inverse bijections 
$$
\begin{array}{rcl}
\Alt_n & \overset{\varphi}{\longrightarrow} & 2^{[n{-}1]}\\
v & \overset{\varphi}{\longmapsto} & S(v):=\{i \in [n{-}1]:  v \cdot \one_{[i{+}1,n]}=+1\} \\
e_1 +\sum_{i \in S} (e_{i+1} - e_i) &  \overset{\varphi^{-1}}{\mapsfrom} & S.
\end{array}
\qedhere
$$
\end{proposition}

\begin{definition} \rm \ \\
Put a partial order $<_{EL}$ on $\Alt_n$ that pulls back the inclusion order on $2^{[n{-}1]}$ via the above
bijection $\varphi$, that is, $v \leq_{EL} w$ if and only if  $S(v) \subseteq S(w)$.
Equivalently, 
$v \leq_{EL} w$ if and only for every $i=1,2,\ldots,n$ one has dot product 
$(w-v) \cdot (e_i+e_{i+1}+\cdots+e_n) \geq 0$.
\end{definition}

\begin{example} \rm \ \\
Here is the order $<_{EL}$ on $\Alt_n$ for $n=3,4,5$:
$$
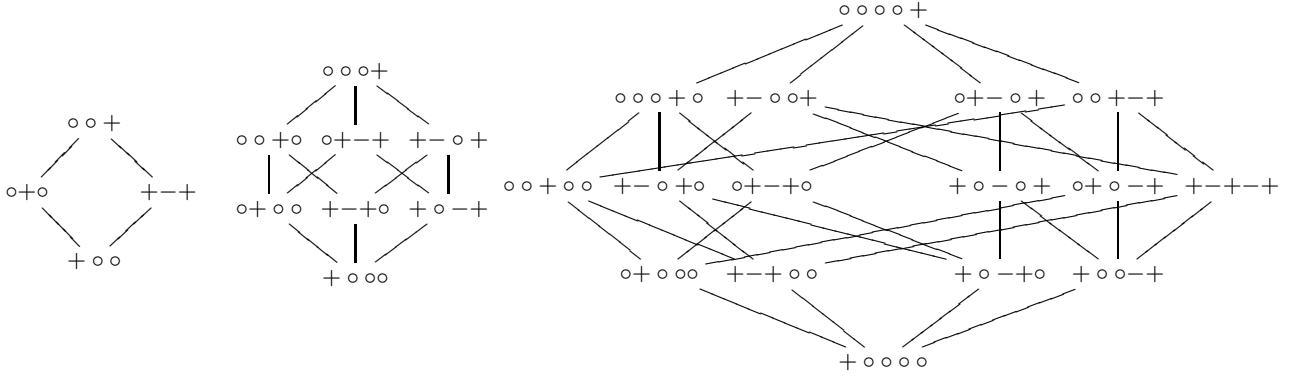
\small
\xymatrix@C=.02in@R=.2in{
& & \\
& & \\
&\circ\circ{+}& \\
\circ{+}\circ\ar@{-}[ur]& &+{-}+\ar@{-}[ul]\\
&{+}\circ\circ\ar@{-}[ur]\ar@{-}[ul]& \\
}
\quad
\xymatrix@C=.02in@R=.2in{
 & & \\
                                &\circ\circ\circ{+}& \\
\circ\circ{+}\circ\ar@{-}[ur]&\circ{+}{-}{+}\ar@{-}[u] &{+}{-}\circ{+}\ar@{-}[ul]\\
\circ{+}\circ\circ\ar@{-}[u]\ar@{-}[ur]&{+}{-}{+}\circ\ar@{-}[ul] \ar@{-}[ur]&{+}\circ{-}{+}\ar@{-}[ul]\ar@{-}[u]\\
                                &{+}\circ\circ\circ \ar@{-}[ul] \ar@{-}[u] \ar@{-}[ur] & \\
}
\xymatrix@C=.03in@R=.3in{
     &                             &                            &\circ\circ\circ\circ{+} &                             &                                & \\
     & \circ\circ\circ+\circ\ar@{-}[urr] & {+}{-}\circ\circ{+} \ar@{-}[ur] &           & \circ{+}{-}\circ{+} \ar@{-}[ul]  &\circ\circ{+}{-}{+}  \ar@{-}[ull]  &\\
\circ\circ{+}\circ\circ\ar@{-}[ur]  \ar@{-}[urrrrr] & {+}{-}\circ{+}\circ \ar@{-}[u] \ar@{-}[ur] & \circ{+}{-}{+}\circ \ar@{-}[ul] \ar@{-}[urr] &   & {+}\circ{-}\circ{+} \ar@{-}[u] \ar@{-}[ull]& \circ{+}\circ{-}{+} \ar@{-}[u] \ar@{-}[ul]  &{+}{-}{+}{-}{+} \ar@{-}[ul]   \ar@{-}[ullll]  \\  
     & \circ{+}\circ\circ\circ\ar@{-}[ul] \ar@{-}[ur] \ar@{-}[urrrr]  & {+}{-}{+}\circ\circ \ar@{-}[ull] \ar@{-}[ul] \ar@{-}[urrrr] &           & {+}\circ{-}{+}\circ \ar@{-}[u]  \ar@{-}[ull] \ar@{-}[ulll] &{+}\circ\circ{-}{+} \ar@{-}[ul] \ar@{-}[u] \ar@{-}[ur]  &\\
      &                             &                            &{+}\circ\circ\circ\circ \ar@{-}[ull] \ar@{-}[ul] \ar@{-}[ur] \ar@{-}[urr]&                             &                                & \\
}
$$
\end{example}

Next, we show that $\lambda:C(\Phi_n) \to \Alt_n$ defined by $\lambda(I \lessdot J):= \one_J - \one_I$
is an EL-labeling of $\Phi_n$ with respect to $<_{EL}$ on $\Alt_n$.  For the rest of this section, 
fix a pair $I <_{\interlace} J$ in $\Phi_n$ with $H_{\min}(I,J)$
as in Definition~\ref{def-of-Hmin}.

\begin{lemma}
\label{l:min_bottom}
Assume $I <_{\interlace} H <_{\interlace} J$ with $\#H=\#I+1$.
Then
$
\one_{H_{\min}(I,J)} - \one_I \,\, \leq_{EL}  \,\, \one_H-\one_I.
$
\end{lemma}

\begin{proof}
Recall $\leq_{EL}$ can be rephrased
as follows: $A \leq_{EL} B$ if and only if 
$(\one_B-\one_A) \cdot \one_{[\ell,n]} \geq 0$ for all $\ell$.

Thus, since $H_{\min}(I,J) \leq_{\comp} H$ according to 
Lemma~\ref{minimal-H-lemma}(iii), for all $\ell$ one will have 
$$
\left( (\one_{H}-\one_{I})-(\one_{H_{\min}(I,J)}-\one_I) \right) \cdot \one_{[\ell,n]}
= (\one_{H}-\one_{H_{\min}(I,J)}) \cdot \one_{[\ell,n]}
\geq 0. \qedhere
$$
\end{proof}

It turns out that one can characterize $H_{\min}(I,J)$ in terms of $\leq_{EL}$.

\begin{lemma}
\label{EL-characterization-of-Hmin}
Assume $I <_{\interlace} H <_{\interlace} J$ with $\#J=\#I+2$.  Then 
$$
\one_{H} - \one_I \,\, \leq_{EL}  \,\, \one_J-\one_K
\quad \text{ if and only if } \quad
H=H_{\min}(I,J).
$$
\end{lemma}

\begin{proof}
Name the elements of $I,H,J$ as follows:
$$
\begin{aligned}
I&=\{i_1 < \cdots < i_p\},\\
H&=\{h_1 < \cdots < h_p<h_{p+1}\},\\
J&=\{j_1 < \cdots < j_p<j_{p+1}<j_{p+2}\}.
\end{aligned}
$$

\vskip.2in
\noindent
($\Leftarrow$):
Assume $H=H_{\min}(I,J)$.  We check for each $\ell$ that 
$
(\one_H-\one_I)\cdot \one_{[\ell,n]} 
\leq (\one_J-\one_H)\cdot \one_{[\ell,n]},
$
or equivalently, 
$$
 \#J \cap [\ell,n] + \#I \cap [\ell,n] - 2\#H\cap [\ell,n] \,\, \geq \,\, 0.
$$
If $H \cap [\ell,n] = \varnothing$, this is clear.  Otherwise, let
$H \cap [\ell,n]=\{h_k,h_{k+1},\ldots,h_{p+1}\}$, so that
$\#H\cap [\ell,n] = p+2-k$.  Then the interlacing $I <_{\interlace} H <_{\interlace} J$
along with $h_k=\max(i_{k-1},j_k)$ imply that
$$
\begin{aligned}
I \cap [\ell,n]&=
\begin{cases}
\{i_k,i_{k+1},\ldots,i_{p}\} & \text{ if }h_k >i_{k-1},\\
\{i_{k-1},i_k,i_{k+1},\ldots,i_{p}\} & \text{ if }h_k =i_{k-1},\\
\end{cases} \\
J \cap [\ell,n]&=
\begin{cases}
\{j_{k+1},j_{k+2},\ldots,i_{p+1}\} & \text{ if }h_k >j_{k},\\
\{j_{k},j_{k+1},j_{k+2},\ldots,j_{p+2}\} & \text{ if }h_k =j_{k}.\\
\end{cases}
\end{aligned}
$$
From this one can calculate that
$$
\#J \cap [\ell,n] + \#I \cap [\ell,n] - 2\#H\cap [\ell,n] =
\begin{cases}
0 & \text{ if }h_k =j_k > i_{k-1} \text{ or } h_k=i_{k-1}>j_k,\\
+1& \text{ if }h_k=i_{k-1}=j_k.
\end{cases}
$$

\vskip.2in
\noindent
($\Rightarrow$):
Assume  
$\one_{H} - \one_I \,\, \leq_{EL}  \,\, \one_J-\one_H.$
\vskip.1in
\begin{quote}
{\bf Claim:} One cannot have both strict inequalities
$i_{k-1} < h_k < i_k$, nor a strict inequality $i_p < h_{p+1}$.
\end{quote}
\vskip.1in
To see this claim, note that in either case 
($i_{k-1} < h_k < i_k$ or $i_p < h_{p+1}$), it would
imply $h_k \in H \setminus I$.  
Then since $I <_{\interlace} H$, this would imply 
$(\one_{H} - \one_I) \cdot \one_{[h_k,n]} =+1$.  But then $h_k \in H$ and
$H <_{\interlace} J$ implies 
$(\one_J-\one_H) \cdot \one_{[h_k,n]} =0 
<+1=(\one_{H} - \one_I) \cdot \one_{[h_k,n]}$,
a contradiction to our assumption.

By Lemma~\ref{minimal-trapezoid-lemma}~(c) and (d), $I<_{\interlace} H <_{\interlace} J$ implies
$h_k \geq \max(i_{k-1},j_k)$ for $k=1,2,\ldots,p+1$.
We must now show that these are all equalities, not inequalities.
For the sake of contradiction, assume not and pick $k$ {\it maximal} 
such that $h_k > \max(i_{k-1},j_k)$.

The Claim above then forces $k \leq p$ and $h_k=i_k$ 
(else $i_{k-1} < h_k < i_k$ or $k=p+1$ and $i_p < h_{p+1}$).
Then $h_{k+1} > h_k=i_k$ and the maximality of $k$ forces
$h_{k+1}=\max(i_k,j_{k+1})=\max(h_k,j_{k+1})=j_{k+1}$.  And again the Claim
forces $k+1 \leq p$ and $i_{k+1}=h_{k+1} (=j_{k+1})$.

We now repeat this argument to show by induction that for all
$m=k+1,k+2,\ldots$, one has both $m \leq p$ and this triple
coincidence $j_m=h_m=i_m$;  this would contradict finiteness of $p$.
The inductive step again notes that
$h_{m+1} > h_m=i_m$ and maximality of $k$ forces 
$h_{m+1}=\max(i_m,j_{m+1})=\max(h_m,j_{m+1})=j_{m+1}$.  But then the
Claim forces $m+1 \leq p$ and $i_{m+1}=h_{m+1}(=j_{m+1})$,
recreating the inductive hypothesis.
\end{proof}

\begin{proof}[Proof of Theorem~\ref{t:asm_el}]
We first check our edge-labeling $\lambda$ satisfies the two conditions for an EL-labeling:
\begin{itemize}
\item for every interval $[x,y] \subset \Phi_n$, there is a unique maximal chain $(x = x_0 \lessdot x_1 \lessdot \dots \lessdot x_k = y)$, that has {\it weakly rising} labels 
$
\lambda(x_{0},x_{1}) \leq_{\Lambda} \lambda(x_{1},x_{2}) \leq_{\Lambda} \cdots  \leq_{\Lambda} \lambda(x_{k-1},x_{x_k})
$
\item if $x \lessdot z < y$, with $z \neq x_1$, then $\lambda(x,x_1) <_{\Lambda} \lambda(x,z)$. 
\end{itemize}
The  first condition follows by combining Lemma~\ref{minimal-trapezoid-lemma}(b)
and Lemma~\ref{EL-characterization-of-Hmin}, which show that  for any $I <_{\interlace} J$,
the unique maximal chain in the interval $[I,J]$ corresponds to the
$(I,J)$-monotone trapezoid $T_{\min}(I,J)$.
Then the second condition comes from Lemma~\ref{l:min_bottom}.

For the second assertion of the theorem,  it suffices to check that if $T, T'$ are monotone triangles with 
$T' <_W T$, then any of the above EL-shellings, which come from linearly extending the lexicographic
ordering of $<_{EL}$ on edge labels, will have $T'$ earlier than $T$.  By definition of the weak order $<_W$,
it suffices to check this holds when
$T' = T \bubble \pi_i$ for some $i$.  In this case, it follows because Lemma~\ref{l:min_bottom} shows that $T$ will
have lexicographically earlier edge label sequence than $T'$:  the two sequences first differ in
replacing the label $\one_{T_{i+1}}-\one_{T_i}$ with the $<_{EL}$-smaller label $\one_{H_{\min}(T_i,T_{i+2})}-\one_{T_i}$.
\end{proof}

\section{Descents, $h$-vectors and flag $h$-vectors}
\label{s:descents}

Recall from the Introduction the usual {\it descent set} for a permutation $w=(w_1,\ldots,w_n)$ in $\symm_n$
$$
\Des(w):=\{k \in [n-1]: w_k> w_{k+1}\} =\{k \in [n-1]: w \bubble \pi_k =ws_k <_W w \}.
$$
It has a natural extension to monotone triangles $T$, motivated by the weak order $<_W$ and our shelling results.

\begin{definition} \rm \ \\
Define the {\it descent set} $\Des(T)$ for $T=(T_0,T_1,\ldots,T_n)$ in $\MT_n$ by
$$
\begin{aligned}
\Des(T)&:=\{k \in [n-1]: T \bubble \pi_k  <_W T \}\\
& =\{ k \in [n-1]: T_k \neq H_{\min}(T_{k-1},T_{k+1}) \}.
\end{aligned}
$$
\end{definition}

There is another way to define $\Des(T)$.
\begin{lemma}
\label{l:max_min_diff}
For $T$  in  $\MT_n$,  one has
$$
 \Des(T): = \{k \in [n-1]: \text{ there does not exist some }T'\neq T \text{ with }T' \bubble \pi_k=T\}.
 $$
In particular, $T$ is one of the maximal elements of the weak order $<_W$ if and only
if $\Des(T)=[n-1]$.
\end{lemma}
\begin{proof}
Since $\pi_k^2=\pi_k$, if there exists $T'$ with $T' \bubble \pi_k =T$, then 
$
T \bubble \pi_k=T' \bubble = T' \bubble \pi_k =T,
$
so $k \not\in \Des(T)$.

Conversely, if $k \not\in \Des(T)$, so that $T \bubble \pi_k=T$, we wish to exhibit at least one $T' \neq T$
having $\pi_k(T')=T$.  From $T \bubble \pi_k=T=(T_0,T_1,\ldots,T_n)$
we know that $T_k=H_{\min}(I,J)$ where $I:=T_{k-1}, J:=T_{k+1}$, so that if we construct
$T'$ from $T$ by replacing $T_k$ with $H_{\max}(I,J)$ as defined in \eqref{H-max-definition}, then
it will certainly have $T' \bubble \pi_k=T$. 

It only remains to show that $T' \neq T$, that is
$H_{\max}(I,J) \neq H_{\min}(I,J)$.  To check this, name elements:
$$
\begin{aligned}
I&=\{i_1 < i_2 <\dots <i_{k-1}\},\\
H_{\min}(I,J) &= \{h_1 < h_2 < \dots < h_{k-1} < h_k\},\\
 H_{\max}(I,J) &= \{h_1' < h_2' < \dots <h_{k-1}'< h_k'\},\\
 J&=\{j_1 < j_2 <\dots <j_{k-1} <j_k <j_{k+1}\}\\
\end{aligned}
$$
Then the formulas defining $H_{\min}(I,J), H_{\max}(I,J)$ are
$
h_m = \max(i_{m-1}, j_m), 
h'_m = \min(i_m, j_{m+1}),
$
implying that $h'_m = h_m$ if and only if $i_m= j_m$ or $i_{m-1} = j_{m+1}$.
Since $\#I \cap J \leq \#I = k-1$, such an equality occurs at most $k-1$ times,
and hence $h_m' \neq h_m$ for at least one $m=1,2,\ldots,k$.
\end{proof}

\begin{remark} \rm \ \\
\label{opposite-action-remark}
Embedded in the previous proof are operators $\pi_k':T \mapsto T'$  on $\MT_n$  for $k=1,2,\ldots,n-1$, 
where $T'$ is obtained from $T$ by replacing $T_k$  
with $T'_k=H_{\max}(T_{k-1},T_{k+1})$.  Because of the relation
between the $H_{\min}$ and $H_{\max}$ constructions
described in Remark~\ref{H-max-remark}, the operators $\{\pi'_k\}_{k=1,2,\ldots,n-1}$ 
satisfy the same braid and quadratic relations
as $\{\pi_k\}$, giving 
a (different) action of the $0$-Hecke monoid $\HHH_n(0)$
on $\MT_n$.  

One can check that this other action, in fact,
extends the {\it (right-)regular action} of 
$\HHH_n(0)$ on itself, when one identifies the monotone triangle
$T(w)$ in $\MT_n$ with $\pi'_w$ in $\HHH_n(0)$.  One could use it
to define a different version of a weak order on $\MT_n$,
having a unique top element $T(w_0)$, but several
different minimal elements.  One reason that we instead chose
the action by $\{\pi_k\}$ and their resulting weak order $<_W$
is so that the monotone triangle $T(e)$ corresponding
to $e=(1,2,\ldots,n)$ in $S_n$ labels the first facet in all of the
shellings.  
\end{remark}

As mentioned in the Introduction, descent sets conveniently encode the {\it flag $f$-vector} 
$f(\Phi_n):=(f_J)_{J \subset [n-1]}$, where $f_J$ counts the number of chains that
pass through the ranks in $J$.
One instead considers the {\it flag $h$-vector} $h(\Phi_n)=(h_J)_{J \subset [n-1]}$, 
defined by these inclusion-exclusion relations:
$$
f_J=\sum_{ I : I \subseteq  J } h_I, \quad \text{ or equivalently, }\quad
h_J=\sum_{ I: I \subseteq J } (-1)^{\#J \setminus I } f_I.
$$
General shelling theory (e.g., Bj\"orner \cite[\S 1(B)]{bjorner1984some}) then implies this combinatorial interpretation for $h_J$:
$$
h_J(\Phi_n)=\# \{ T \in \MT_n: \Des(T)=J \}
$$
The usual $f$-vector $f=(f_{-1},f_0,f_1,\ldots,f_{n-2})$ and $h$-vector $h=(h_0,h_1,\ldots,h_{n-1})$
for $\Delta(\Phi_n\setminus\{\hat{0},\hat{1}\})$
can then be obtained by grouping the terms in $(f_J), (h_J)$ as follows:
$$
\begin{aligned}
f_i =\sum_{ J \in \binom{[n-1]}{i+1}:  } f_J, \quad \text{ and } \quad
h_i &=\sum_{ J  \in \binom{[n-1]}{i}:  } h_J.
\end{aligned}
$$
In particular, $h_i(\Phi_n)=\# \{ T \in \MT_n: \#\Des(T)=i\}$.  See Table~\ref{t1} for the $h$-vector $h(\Phi_n)$
and flag $h$-polynomial for small values of $n$.

\begin{table}
\begin{center}
\begin{tabular}{|c||l|}\hline
$n$ & $h(\Phi_n)=(h_0,h_1,\ldots,h_{n-1})$ \\
\hline \hline
$2$ & $(1,1)$ \\
\hline
$3$ & $(1,4,2)$\\
\hline
$4$ & $(1,11,21,9)$\\
\hline
$5$ & $(1,26,130,192,80)$\\
\hline
$6$ & $(1,57,638,2318,3101,1321)$ \\
\hline
$7$ & $(1, 120, 2773, 21472, 67340, 87616, 39026)$ \\
\hline
$8$ & $(1, 247, 11264, 172222, 1108243, 3260759, 4280764, 2016716)$\\ \hline
\end{tabular}

\vskip.2in

\begin{tabular}{|c||l|}\hline
& \\
$n$ & $\sum_{J \subset [n-1]} h_J(\Phi_n) \,\,x_J$ where $x_J:=\prod_{i \in J} x_i$  \\
&\\
\hline\hline
$2$ & $ x_1 + 1$\\
\hline
$3$ &  $ 2x_1x_2 + 2x_1 + 2x_2 + 1$\\
\hline
$4$ & $ 9 x_1 x_2 x_3 + 7 x_1 x_2 + 7 x_1 x_3 + 7 x_2 x_3 + 3 x_1 + 5 x_2 + 3 x_3 + 1$\\
\hline
$5$ &  $80 x_1 x_2 x_3 x_4 + 52 x_1 x_2 x_3 + 44 x_1 x_2 x_4 + 44 x_1 x_3 x_4+ 52 x_2 x_3 x_4 + 16 x_1 x_2 + 26 x_1 x_3 + 32 x_2 x_3$\\
&$ + 14 x_1 x_4 + 26 x_2 x_4 + 16 x_3 x_4 + 4 x_1 + 9 x_2 + 9 x_3 + 4 x_4 + 1$\\
\hline
$6$ & $1321 x_1 x_2 x_3 x_4 x_5 + 745 x_1 x_2 x_3 x_4 + 562 x_1 x_2 x_3 x_5 + 487 x_1 x_2 x_4 x_5 + 562 x_1 x_3 x_4 x_5 + 745 x_2 x_3 x_4 x_5$\\
& $ + 180 x_1 x_2 x_3 + 251 x_1 x_2 x_4 + 298 x_1 x_3 x_4 + 405 x_2 x_3 x_4 + 120 x_1 x_2 x_5 + 215 x_1 x_3 x_5 + 298 x_2 x_3 x_5$ \\
& $ + 120 x_1 x_4 x_5 + 251 x_2 x_4 x_5 + 180 x_3 x_4 x_5 + 30 x_1 x_2 + 65 x_1 x_3 + 92 x_2 x_3 + 58 x_1 x_4 + 125 x_2 x_4 $\\
& $ + 92 x_3 x_4 + 23 x_1 x_5 + 58 x_2 x_5+ 65 x_3x_5 + 30x_4x_5 + 5x_1 + 14x_2 + 19x_3 + 14x_4 + 5x_5 + 1$\\ \hline
\end{tabular}
\end{center}

\vskip.2in

\caption{\label{t1} 
The $h$-vectors of $\Phi_n$ for $n \leq 8$ and flag $h$-polynomials of $\Phi_n$ for $n \leq 6$.
All data computed using {\tt SAGE}.
}
\end{table}

We remark on some features of this data.
Note the sequence of values  $1,2,9,80, 1321, 39026,2016716$ for
$$
h_{n-1} =\#\{T \in \MT_n: \Des(T)=[n-1]\} = \#\{ \text{ maximal elements in the poset }(\MT_n,<_W) \},
$$
appearing at the right in Table~\ref{t1},
which is not in the Online Encyclopedia of Integer Sequences ({\tt OEIS}).

The data invites comparison with the Boolean algebra $2^{[n]}$, which
has $h$-vector $h(2^{[n]})=(h_0,h_1,\ldots,h_{n-1})$ 
given by the {\it Eulerian numbers}, that is, 
$h_i(2^{[n]})=\#\{w \in \symm_n: \#\Des(w)=i\}$.
The Eulerian numbers are well-behaved in many ways (see Petersen \cite{Petersen}).
For example, they satisfy recurrences and have the {\it symmetry} $h_i = h_{n-1-i}$.
They also have the very strong property that the {\it $h$-polynomial} 
$$
h(2^{[n]},t):=\sum_{i=0}^{n-1} h_i t^i=\sum_{w \in \symm_n} t^{\#\Des(w)}
$$
has only real zeroes.  This implies
{\it log-concavity} $h_i ^2 \geq h_{i+1} h_{i-1}$, which then implies {\it unimodality}, meaning
that there is some $k$ (in this case $k=\lfloor{\frac{n-1}{2}} \rfloor$ works) for which $h_0 \leq h_1 \leq \cdots h_k \geq \cdots \geq h_{n-2} \geq h_{n-1}$.
From the data in Table~\ref{t1}, the reader can check that for $\Phi_n$,
the $h$-polynomial 
$$
h(\Phi_n,t):=\sum_{i=0}^{n-1} h_i t^i=\sum_{T \in \MT_n} t^{\#\Des(T)}
$$ 
is irreducible in $\Q[t]$ with only real zeroes for $n \leq 8$, hence is log-concave for those values.

\begin{question} 
\label{q:log-concavity}
Does $h(\Phi_n,t)$ have only real zeroes?
If not, is its coefficient sequence log-concave, or at least unimodal?
\end{question}

\begin{question}
\label{q:second}
What is the largest entry in the $h$-vector of $\Phi_n$?
Is it always $h_{n-2}$?
\end{question}

%


\section{Descents as a map to $\QSYM$, and proof of Theorem~\ref{t:alg_map}}
\label{s:qsym}

As described in the Introduction, the map $w \mapsto \Des(w)$ that sends a permutation $w$ in $\symm_n$ 
to its descent set was pleasingly reinterpreted in the work of Malvenuto and Reutenauer \cite{malvenuto1995duality} as
a morphism of Hopf algebras.  We wish to explain here how this extends to the map $T \mapsto \Des(T)$
sending a monotone triangle to its descent set, giving at least an algebra (but not coalgebra) morphism out
of the Hopf algebra of $ASM$s recently defined by Cheballah, Giraudo and Maurice~\cite{cheballah2015hopf}.

Let us start by recalling the algebra structures on quasisymmetric functions, permutations, and $ASM$s.

\begin{definition} \rm \ \\
The  {\it ring of quasisymmetric functions} $\QSym$ can be defined as the subalgebra of the algebra $\Z[[x_1,x_2,\ldots]]$ of formal power series
 that has $\Z$-basis given by the {\it monomial quasisymmetric functions} 
$$
M_\alpha:=\sum_{1 \leq i_1 < i_2 <\ldots <i_k} x_{i_1}^{\alpha_1} \cdots x_{i_k}^{\alpha_k}
$$
as $\alpha=(\alpha_1,\ldots,\alpha_k)$ runs through all (ordered) compositions having $\alpha_i \in \{1,2,\ldots\}$ and
any length $k \geq 0$.
\end{definition}

The ring $\QSym$ was introduced by Gessel \cite{Gessel}.  He observed that if one defines the unitriangularly related
$\Z$-basis of {\it fundamental quasisymmetric functions}
\begin{equation}
\label{fundamental-quasisymm-function-definition}
L_\alpha := \sum_{\beta: \beta \text{ coarsens }\alpha} M_\beta
\end{equation}
then results from Stanley's theory of {\it $P$-partitions} \cite[Cor. 7.19.5]{Stanley-EC2} imply the following expansion
for products of $L_\alpha$'s.  Given a subset $J=\{j_1 < \dots <j_\ell\} \subseteq [n-1]$, define its {\it associated composition} of $n$
to be  
$$
\alpha(J):=(j_1,j_2-j_1,j_3-j_2, \ldots,j_\ell-j_{\ell-1},n-j_{\ell}).
$$
In other words, $\alpha(J)$ is the composition whose partial sums are the elements of $J$.

For $u,v$ in $\symm_a ,\symm_b$, let $u \shuffle v[a]$ be the set of all shuffles $w=(w_1,w_2,\ldots,w_{a+b})$ of the sequences $u=(u_1,\ldots,u_a)$, and
$v[a]=(a+v_1,a+v_2,\ldots,a+v_b)$.
In other words, $w \in \symm_{a+b}$ is in $u \shuffle v[a]$ if $(u_1,\dots,u_a)$ and $(a+v_1,\dots,a+v_b)$ are subsequences of $w$.
\begin{proposition}
Given $u,v$ in $\symm_a ,\symm_b$, 
\begin{equation}
\label{Qsym-shuffle-product-rule}
L_{\alpha(\Des(u))} \cdot L_{\alpha(\Des(v))} = \sum_{w\,\,  \in \,\, u \,\, \shuffle\,\, v[a]} L_{\alpha(\Des(w))}.
\end{equation}
\end{proposition}

This was part of Malvenuto and Reutenauer's motivation for the following definition.
\begin{definition} \rm \ \\
The {\it Malvenuto-Reutenauer (Hopf) algebra of permutations} is a graded free abelian
group 
$$
\FQSym=\bigoplus_{n \geq 0} \FQSym_n,
$$
in which $\FQSym_n$ has $\Z$-basis elements $\{\textbf{w}\}_{w \in \symm_n}$.
As an algebra, its multiplication is extended $\Z$-linearly from this rule: for $u,v$ in $\symm_a ,\symm_b$,
\begin{equation}
\label{FQSym-multiplication-rule}
\textbf{u} \cdot \textbf{v} = \sum_{w\,\,  \in \,\, u \,\, \shuffle\,\, v[a]} \textbf{w}
\end{equation}
in which the sum runs over the same set of $w$ as in \eqref{Qsym-shuffle-product-rule}.
\end{definition}

Thus the algebra structure on $\FQSym$ was defined so that this map is a (surjective) algebra morphism: 
\begin{equation} 
\label{phi-on-FQSym-is-an-algebra-map}
\begin{array}{rcl}
\FQSym &\overset{\varphi}{\longrightarrow}& \QSym\\
\textbf{w} &\longmapsto& L_{\alpha(\Des(w))}
\end{array}
\end{equation}

\begin{definition} \rm \ \\
Cheballah, Giraudo and Maurice ~\cite{cheballah2015hopf} embedded $\FQSym$ inside a larger 
graded Hopf algebra 
\begin{equation}
\mathcal{ASM}=\bigoplus_{n \geq 0} \mathcal{ASM}_n,
\end{equation}
whose $n^{th}$-graded component $\mathcal{ASM}_n$ has a 
$\Z$-basis $\{\textbf{A}\}$ indexed by $A$ in $\ASM_n$.  Its algebra structure generalizes that of $\FQSym$ to the following
{\it row-shuffle}\footnote{Actually, in \cite{cheballah2015hopf} the algebra structure uses column shuffles, but this is equivalent to what is described here
after transposing the alternating sign matrices $A \mapsto A^t$.} product.
Given ASMs $A, B$ of sizes $a \times a$ and $b \times b$, define $A\circ b$ to be the $a \times (a+b)$ matrix with first $a$ columns $A$ and last $b$ columns all $0$-vectors.
Likewise, $a \circ B$ is the $b \times (a+b)$ matrix with last $b$ columns $B$ and first $a$ columns all $0$-vectors.
Then define 
\begin{equation}
\label{CGM-product-formula}
\textbf{A} \cdot \textbf{B}= \sum_{C\,\,  \in \,\, (A\circ b) \,\, \shuffle \,\, (a \circ B)} \textbf{C}
\end{equation}
where $C$ runs through all the $(a+b) \times (a+b)$ matrices obtained by shuffling the rows of $A\circ b$ and of
$a \circ B$.   
\end{definition}

\begin{example} \rm \ \\
\label{CGM-product-example}
If $A=\left[\begin{smallmatrix}
0 & + & 0 \\
+ & - & +\\
0 & + & 0
\end{smallmatrix}\right]$
and 
$B=\left[\begin{smallmatrix}
0 & +\\
+ & 0
\end{smallmatrix}\right],
$
then 
$
A \circ b 
=\left[\begin{smallmatrix}
\bf{0} & \bf{+} & \bf{0} & 0 & 0 \\
\bf{+} &\bf{-} & \bf{+} & 0 & 0 \\
\bf{0} & \bf{+} & \bf{0} & 0 & 0 \end{smallmatrix}\right],
$
and
$
a \circ B 
=\left[\begin{smallmatrix}
0& 0& 0& \bf{0} & \bf{+}\\
0& 0& 0& \bf{+} & \bf{0}
\end{smallmatrix}\right].
$
One then has 
$$ 
\begin{aligned}
\textbf{A} \cdot \textbf{B}
&=
\left[\begin{smallmatrix}
0 & + & 0 \\
+ & - & +\\
0 & + & 0
\end{smallmatrix}\right]
\cdot \left[\begin{smallmatrix}
0 & +\\
+ & 0
\end{smallmatrix}\right]
\\
&=
\left[\begin{smallmatrix}
0 & + & 0 & 0 & 0\\
+ & - & + & 0 & 0\\
0 & + & 0 & 0 & 0\\
0 & 0 & 0 & 0 &+\\
0 & 0 & 0 & + & 0
\end{smallmatrix}\right]
+\left[\begin{smallmatrix}
0 & + & 0 & 0 & 0\\
+ & - & + & 0 & 0\\
0 & 0 & 0 & 0 & +\\
0 & + & 0 & 0 & 0\\
0 & 0 & 0 & + & 0
\end{smallmatrix}\right]
+\left[\begin{smallmatrix}
0 & + & 0 & 0 & 0\\
+ & - & + & 0 & 0\\
0 & 0 & 0 & 0 & +\\
0 & 0 & 0 & + & 0\\
0 & + & 0 & 0 & 0
\end{smallmatrix}\right]
+\left[\begin{smallmatrix}
0 & + & 0 & 0 & 0\\
0 & 0 & 0 & 0 & +\\
+ & - & + & 0 & 0\\
0 & + & 0 & 0 & 0\\
0 & 0 & 0 & + & 0
\end{smallmatrix}\right]
+\left[\begin{smallmatrix}
0 & + & 0 & 0 & 0\\
0 & 0 & 0 & 0 & +\\
+ & - & + & 0 & 0\\
0 & 0 & 0 & + & 0\\
0 & + & 0 & 0 & 0
\end{smallmatrix}\right]\\
&\ 
+\left[\begin{smallmatrix}
0 & + & 0 & 0 & 0\\
0 & 0 & 0 & 0 & +\\
0 & 0 & 0 & + & 0\\
+ & - & + & 0 & 0\\
0 & + & 0 & 0 & 0
\end{smallmatrix}\right]
+\left[\begin{smallmatrix}
0 & 0 & 0 & 0 & +\\
0 & + & 0 & 0 & 0\\
+ & - & + & 0 & 0\\
0 & + & 0 & 0 & 0\\
0 & 0 & 0 & + & 0
\end{smallmatrix}\right]
+\left[\begin{smallmatrix}
0 & 0 & 0 & 0 & +\\
0 & + & 0 & 0 & 0\\
+ & - & + & 0 & 0\\
0 & 0 & 0 & + & 0\\
0 & + & 0 & 0 & 0
\end{smallmatrix}\right]
+\left[\begin{smallmatrix}
0 & 0 & 0 & 0 & +\\
0 & + & 0 & 0 & 0\\
0 & 0 & 0 & + & 0\\
+ & - & + & 0 & 0\\
0 & + & 0 & 0 & 0
\end{smallmatrix}\right]
+\left[\begin{smallmatrix}
0 & 0 & 0 & 0 & +\\
0 & 0 & 0 & + & 0\\
0 & + & 0 & 0 & 0\\
+ & - & + & 0 & 0\\
0 & + & 0 & 0 & 0
\end{smallmatrix}\right].
\end{aligned}
$$
\end{example}

Note that when one restricts the product
 formula \eqref{CGM-product-formula} to the elements of the form $\textbf{w}:=\textbf{{A(w)}}$ where $A(w)$ is the
permutation matrix corresponding to $w^{-1}$, it agrees with the multiplication rule for $\textbf{u} \cdot \textbf{v}$
given in \eqref{FQSym-multiplication-rule}.
We also wish to recast the formula \eqref{CGM-product-formula} in terms of monotone triangles.  
The following proposition is straightforward
using the bijection $\ASM_n \rightarrow \MT_n$ described in the Introduction.

\begin{proposition}
\label{row-shuffles-in-terms-of-triangles}
Fix $A, B$ in $\ASM_a, \ASM_b$, with corresponding monotone
triangles $T(A), T(B)$ in $\MT_a, \MT_b$.
Let $C$  in   $(A\circ b) \,\, \shuffle \,\, (a \circ B)$
have 
\begin{itemize}
\item 
$S \subset [a+b]$ the $a$-element subset indexing the rows of $C$ that come from
$A\circ b$, and
\item
$[a+b] \setminus S$ the $b$-element subset indexing the rows of $C$ that come from $a \circ B$.
\end{itemize}
Then $T(C)$ in $\MT_{a+b}$ has as its $k^{th}$ row the set
$$
T(C)_k=T(A)_i \sqcup (a+T(B))_j,
$$
where 
\begin{itemize}
\item
$i=\#S \cap [k]$, and
\item
$j=\#([a+b] \setminus S)\cap [k] \,\, \left( = k-i\right) $.
\end{itemize}
\end{proposition}

\begin{example} \rm \ \\
For $A=\left[\begin{smallmatrix}
0 & + & 0 \\
+ & - & +\\
0 & + & 0
\end{smallmatrix}\right]$
and 
$B=\left[\begin{smallmatrix}
0 & +\\
+ & 0
\end{smallmatrix}\right]
$
as in  Example~\ref{CGM-product-example},
one has 
$T(A)=\xymatrix@R=.01in@C=.01in{
*\txt<3pc>{2\\1 3\\1 2 3}}$
and 
$a+T(B)=\xymatrix@R=.01in@C=.01in{
*\txt<2pc>{5\\4 5}}$
Hence the terms $\textbf{C}$ appearing in the product $\textbf{A} \cdot \textbf{B}$ correspond to these monotone triangles $T(C)$: 

\begin{center}
\begin{tabular}{|c|c|c|c|c|c|}\hline
$S$ & $\{1,2,3\}$ &$\{1,2,4\}$ &$\{1,2,5\}$ &$\{1,3,4\}$ &$\{1,3,5\}$ \\ \hline\hline
 & & & & & \\ 
$T(C)$&
\xymatrix@R=.01in@C=.01in{
*\txt<4pc>{2\\1 3\\1 2 3\\1 2 3 5}}
& \xymatrix@R=.01in@C=.01in{
*\txt<4pc>{2\\1 3\\1 3 5\\1 2 3 5}}
&\xymatrix@R=.01in@C=.01in{
*\txt<4pc>{2\\1 3\\1 3 5\\1 3 4 5}}
&\xymatrix@R=.01in@C=.01in{
*\txt<4pc>{5\\2 5\\1 3 5\\1 2 3 5}}
&\xymatrix@R=.01in@C=.01in{
*\txt<4pc>{5\\2 5\\1 3 5\\1 3 4 5}}\\ 
 & & & & & \\ \hline
\end{tabular}
\vskip.1in
\begin{tabular}{|c|c|c|c|c|c|}\hline
$S$ & $\{1,4,5\}$ &$\{2,3,4\}$ &$\{2,3,5\}$ &$\{2,4,5\}$ &$\{3,4,5\}$ \\ \hline\hline
 & & & & & \\ 
$T(C)$&\xymatrix@R=.01in@C=.01in{
*\txt<4pc>{2\\2 5\\2 4 5\\1 3 4 5}}
&\xymatrix@R=.01in@C=.01in{
*\txt<4pc>{5\\2 5\\1 3 5\\1 2 3 5}}
&\xymatrix@R=.01in@C=.01in{
*\txt<4pc>{5\\2 5\\1 3 5\\1 3 4 5}}
&\xymatrix@R=.01in@C=.01in{
*\txt<4pc>{2\\2 5\\2 4 5\\1 3 4 5}}
&\xymatrix@R=.01in@C=.01in{
*\txt<4pc>{2\\4 5\\2 4 5\\1 3 4 5}}\\ 
 & & & & & \\ \hline
\end{tabular}
\end{center}
\end{example}

Recall the statement of the theorem.
\vskip.1in
\noindent
{\bf Theorem} \ref{t:alg_map}.
{
The map $\FQSym \overset{\varphi}{\longrightarrow} \QSym$ in \eqref{FQsym-to-QSym-map} 
extends to an algebra (but not a coalgebra) morphism 
$$
\begin{array}{rcl}
\mathcal{ASM} &\overset{\varphi}{\longrightarrow}& \FQSym\\
\textbf{A} &\longmapsto& L_{\alpha(\Des(A))}
\end{array}
$$ 
where $\Des(A)=\Des(T(A))$ for $A$ in $\ASM_n$ is the descent set of its monotone triangle $T(A)$.
}
\vskip.1in

\begin{proof}[Proof of Theorem~\ref{t:alg_map}]
Given $A, B$ in $\ASM_a, \ASM_b$, 
we claim that the multiset of descent sets $\{\Des(T(C))\}$ as $C$ runs through the elements
of $(A \circ b) \,\, \shuffle (a \circ B)$
depends only upon the descent sets $\Des(T(A)), \Des(T(B))$, not on $A, B$ themselves.  Assuming this claim for the
moment, one finishes the proof by picking arbitrary $u, v$ in $\symm_a, \symm_b$
having $\Des(u)=\Des(T(A))$ and $\Des(v)=\Des(T(B))$, and calculating
$$
\begin{aligned}
\varphi(\textbf{A} \cdot \textbf{B}) 
=\sum_{C \in (A \circ b) \,\, \shuffle \,\, (a \circ B)} L_{\alpha(\Des(T(C)))} 
&=\sum_{w \in u \shuffle v} L_{\alpha(\Des(w))} \\
&=L_{\alpha(\Des(u))} L_{\alpha(\Des(v))}  
=L_{\alpha(\Des(A))} L_{\alpha(\Des(B))} 
=\varphi(\textbf{A}) \varphi(\textbf{B}). 
\end{aligned}
$$
Here the second equality used the claim, while the third equality used  \eqref{phi-on-FQSym-is-an-algebra-map}.

To prove the claim, note that each $C$ in $(A \circ b) \,\, \shuffle (a \circ B)$
is determined by the $a$-subset $S \subset [a+b]$ indexing the rows of $C$ that come from $A \circ b$.
We give rules in cases below that decide whether some 
$k \in [a+b-1]$ lies in $\Des(T(C))$, 
based only on the subset $S$ and the descent sets $\Des(T(A))$ and $\Des(T(B))$, not
on $A, B$ themselves.  As notation, let $i:=\#S \cap [k-1], j:=\#(([a+b] \setminus S) \cap [k-1]))$, 
and name these elements:
$$
\begin{aligned}
T(A)_{i+2}=\{a_1 < \cdots < a_i < a_{i+1} < a_{i+2} \},\\
(a+T(B))_{j+2}=\{b_1 < \cdots < b_j < b_{j+1} < b_{j+2} \}.\\
\end{aligned}
$$
Note that deciding whether $k$ lies in $\Des(T(C))$ simply
means checking whether any of the entries of $T_k'$,
where $T':=(T(C) \bubble \pi_k)_k=H_{\min}(T(C)_{k-1},T(C)_{k+1})$, {\it differs} from
the corresponding entry of $T(C)_k$, when computed via
the formula \eqref{Hmin-formula} as the maximum of its two 
neighboring entries to the northwest and southwest.

\vskip.1in
\noindent
{\sf Case 1.} Both $k, k+1$ lie in $S$.\\
In this case, Proposition~\ref{row-shuffles-in-terms-of-triangles} implies
that $(T(C)_{k-1}, T(C)_k, T(C)_{k+1})$ look like this:
$$
\begin{array}{cccccccccccccc}
   &       &a_1   &\cdots&a_i    &       &b_1      &\cdots   &b_j   &      &   \\
   &a_1    &\cdots&a_{i} &       &a_{i+1}&         &   b_1   &\cdots&b_j   &   \\
a_1&\cdots &a_{i} &      &a_{i+1}&       &a_{i+2}  &         &b_1   &\cdots&b_j
\end{array}
$$
Each of the entries $b_m$ in $T(C)_k$ equals its northwest neighbor, so is
unchanged in $T(C) \bubble \pi_k$.  This implies that $k \in \Des(T(C))$ if and only if $k \in \Des(T(A))$.

\vskip.1in
\noindent
{\sf Case 2.}  Both $k, k+1$ lie in $[a+b] \setminus S$.\\
Here $(T(C)_{k-1}, T(C)_k, T(C)_{k+1})$ look like this:
$$
\begin{array}{cccccccccccccc}
   &       &a_1   &\cdots&a_i    &       &b_1      &\cdots   &b_j   &      &   \\
   &a_1    &\cdots&a_{i} &       &b_1&   \cdots   &b_j&     &b_{j+1}   &   \\
a_1&\cdots &a_{i}        &  &b_1&    \cdots  &b_j   & &b_{j+1}& &b_{j+2}
\end{array}
$$
Similarly to Case 1, each entry $a_m$ in $T(C)_k$ equals its southwest neighbor, so is
unchanged in $T(C) \bubble \pi_k$.  This implies $k \in \Des(T(C))$ if and only if $k \in \Des(T(B))$.

\vskip.1in
\noindent
{\sf Case 3.}  $k$ lies in $S$, but $k+1$ lies in $[a+b] \setminus S$.\\
Here $(T(C)_{k-1}, T(C)_k, T(C)_{k+1})$ look like this:
$$
\begin{array}{cccccccccccccc}
   &       &a_1   &\cdots&a_i    &       &b_1      &\cdots   &b_j   &      &   \\
   &a_1    &\cdots&a_{i} &       &a_{i+1}&         &   b_1   &\cdots&b_j   &   \\
a_1&\cdots &a_{i}        &  &a_{i+1}&    &b_1   &\cdots &b_j& &b_{j+1}
\end{array}
$$
We claim that in this case, $k \not\in \Des(T(C))$, since
each entry $a_m$ of $T(C)_k$ equals its southwest neighbor, while
each entry $b_m$ of $T(C)_k$ equals its northwest neighbor.

\vskip.1in
\noindent
{\sf Case 4.}  $k+1$ lies in $S$, but $k$ lies in $[a+b] \setminus S$.\\
Here $(T(C)_{k-1}, T(C)_k, T(C)_{k+1})$ look like this:
$$
\begin{array}{cccccccccccccc}
   &       &a_1   &\cdots&a_i    &       &b_1      &\cdots   &b_j   &      &   \\
   &a_1    &\cdots&a_{i} &       &b_1& \cdots   &b_j& &b_{j+1}   &   \\
a_1&\cdots &a_{i}        &  &a_{i+1}&    &b_1   &\cdots &b_j& &b_{j+1}
\end{array}
$$
In this case $k \in \Des(T(C))$, since  
the entry $b_1$ of $T(C)_k$ has $b_1 > a \geq \max(a_i,a_{i+1})$.

To see that  $\textbf{A} \overset{\varphi}{\longmapsto} L_{\alpha(\Des(A))}$
is not a coalgebra morphism, for example,
one can check from the coproduct formula of Cheballah, Giraudo and Maurice \cite[(1.3.5)]{cheballah2015hopf} that the alternating sign matrix 
$
A = \left[\begin{smallmatrix}
0 & +1 & 0 \\
+1 & -1 & +1\\
0 & +1 & 0
\end{smallmatrix}\right]
$
has coproduct $\Delta(\textbf{A}) = 1 \otimes \textbf{A} + \textbf{A} \otimes 1$,
that is, $\textbf{A}$ is primitive.  Meanwhile, its image $\varphi(\textbf{A})=L_{(1,1,1)}$
has 
$$
\Delta(\varphi(\textbf{A}))= \Delta(L_{(1,1,1)})
=1 \otimes L_{(1,1,1)} 
+ L_{(1)} \otimes L_{(1,1)} 
+ L_{(1,1)} \otimes L_{(1)}
+ L_{(1,1,1)} \otimes 1,
$$
which is not the same as
$
(\varphi \otimes \varphi) (\Delta(\textbf{A})) = 1 \otimes L_{(1,1,1)} 
+ L_{(1,1,1)} \otimes 1.
$
That is, $\varphi(\textbf{A})$ is {\it not} primitive.
\end{proof}

\begin{remark} \rm \ \\
It is well-known, and not hard to see (e.g., as a special case of \cite[Thm. 7.19.7]{Stanley-EC2}),
that applying $\varphi$ to the sum of all of the basis elements $\{\textbf{w}\}_{w \in S_n}$ for 
$\FQSym_n$ gives a readily-identifiable symmetric function
$$
\varphi\left( \sum_{w \in \symm_n} \textbf{w} \right)
=\sum_{w \in \symm_n} L_{\alpha(\Des(w))}
=(x_1+x_2+\cdots)^n.
$$

This fails for $\ASM_n$, e.g., the data in Table~\ref{t1} 
for $n=4$ together with \eqref{fundamental-quasisymm-function-definition} shows that
$$
\begin{aligned}
\varphi\left( \sum_{A \in \ASM_n} \textbf{A} \right)
&=\sum_{A \in \ASM_4} L_{\alpha(\Des(T(A)))} \\
&=L_{(4)}
+3 L_{(1,3)}
+5 L_{(2,2)}
+3 L_{(3,1)}
+7 L_{(1,1,2)}
+7 L_{(1,2,1)}
+7 L_{(2,1,1)}
+9 L_{(1,1,1,1)} \\
&=M_{(4)}
+4 M_{(1,3)}
+6 M_{(2,2)}
+4 M_{(3,1)}
+16 M_{(1,1,2)}
+14 M_{(1,2,1)}
+16 M_{(2,1,1)}
+42 M_{(1,1,1,1)} \\
\end{aligned}
$$
which is not a symmetric function, because its coefficient on $M_\alpha$ is
not constant for all compositions $\alpha$ within the same rearrangement class.
It would be interesting to find natural subcollections $\{ A \}$ 
of $\ASM_n$, not contained entirely in $\symm_n$,
for which $\varphi\left( \sum_A \textbf{A} \right)$ is a symmetric function.

\end{remark}

\section{Poset properties of weak order on $\MT_n$}
\label{pathology-section}

The weak order $<_W$ on the symmetric group $\symm_n$ has many pleasant poset-theoretic properties:
\begin{itemize}
\item It has bottom and top elements 
$\hat{0}=e=(1,2,\ldots,n-1,n)$ and $\hat{1}=w_0=(n,n-1,\ldots,2,1)$.
\item It is a {\it lattice}.
\item It is {\it ranked} with rank function
given by the cardinality $\#\Inv(w)$ of the {\it (left-)inversion set}
of $w$:
$$
\Inv(w)=\{ (w_j,w_i): 1 \leq i < j \leq n \text{ and } w_i > w_j \}.
$$
\item It has an encoding via {\it inclusion} of these (left-)inversion sets:
$u <_W v$ if and only if $\Inv(u) \subset \Inv(v)$.
\item The {\it M\"obius function} $\mu(u,v)$ for $u <_W v$ only takes on values in $\{0,+1,-1\}$.
\item More precisely, the {\it homotopy type} of the order complex $\Delta(u,v)$ 
of any of its open intervals $(u,v)$
is {\it contractible} or {\it homotopy-spherical}.  Specifically, one can phrase this in terms of
$\HHH_n(0)$-action on $S_n$ as follows: $\Delta(u,v)$ is contractible unless
$u = v \bubble \pi_{w_0(J)}$ for some subset $J \subset \Des(u)$, in which case, $\Delta(u,v)$
is homotopy-equivalent to a $(\#J-2)$-dimensional sphere; 
see Bj\"orner \cite[Theorem 6]{bjorner-orderings}.
\end{itemize}

\vskip.1in
Only a few of these properties extend to the weak order $<_W$ to $\MT_n$.
It is still true that $(\MT_n,<_W)$ has a bottom element 
$\hat{0}=T(e)=(\varnothing,\{1\},\{1,2\},\{1,2,3\},\ldots,[n])$,
but it no longer has a top element $\hat{1}$, as there are
many maximal elements.

Since $\MT_n$ is finite, and has no top element, it cannot be a {\it lattice},
but it is also true that its intervals fail to be lattices.  For example,
the lower interval shown on the left in Figure~\ref{f1} is not a lattice, because, for example,
$$
\xymatrix@R=.01in@C=.01in{
*\txt<3pc>{1\\1 3\\1 2 3}}
\text{ and }
\xymatrix@R=.01in@C=.01in{
*\txt<3pc>{2\\1 2\\1 2 3}}
\text{ 
do not have a least upper bound since both 
}
\xymatrix@R=.01in@C=.01in{
*\txt<3pc>{2\\1 3\\1 2 3}}
\text{ and }
\xymatrix@R=.01in@C=.01in{
*\txt<3pc>{3\\2 3\\1 2 3}}
$$
are minimal upper bounds.
Note that this same lower interval is not ranked since there are maximal chains of lengths four and five.

\begin{figure}
$$
\xymatrix@R=.2in@C=.3in{ 
&  &*\txt<3pc>{  3  \\ 2 3 \\1 3 4}\ar^{\pi_1}[dr]\ar^{\pi_3}[dd]\ar_{\pi_2}[ddll]& & \\
& & &*\txt<3pc>{  2  \\ 2 3 \\1 3 4}\ar_{\pi_2}[d]\ar^{\pi_3}[ddrr]& \\
*\txt<3pc>{  3  \\ 1 3 \\1 3 4}\ar_{\pi_3}[d]\ar^<{\pi_1}[drr]
& &*\txt<3pc>{  3  \\ 2 3 \\1 2 3}\ar^{\pi_1}[drrr]\ar_<{\pi_2}[dll]
 & *\txt<3pc>{  2  \\ 1 3 \\1 3 4}\ar_{\pi_3}[d]\ar_{\pi_1}[dl]&\\
*\txt<3pc>{  3  \\ 1 3 \\1 2 3}\ar_{\pi_1}[drr]&&
 *\txt<3pc>{  1  \\ 1 3 \\1 3 4}\ar_{\pi_3}[d]&*\txt<3pc>{  2  \\ 1 3 \\1 2 3}\ar^{\pi_1}[dl] \ar_{\pi_2}[dr]& &*\txt<3pc>{  2  \\ 2 3 \\1 2 3} \ar^{\pi_2}[dl]\\
& & *\txt<3pc>{  1  \\ 1 3 \\1 2 3}\ar_{\pi_2}[dr]& &*\txt<3pc>{  2  \\ 1 2 \\1 2 3}\ar^{\pi_1}[dl]\\
& & & *\txt<3pc>{  1  \\ 1 2 \\1 2 3}& & 
} \quad
\xymatrix@R=.3in@C=.01in{ 
 y=&*\txt<3pc>{  3  \\ 2 3 \\1 3 4}\ar^{\pi_1}[drr]\ar^{\pi_2}[dd]\ar_{\pi_3}[ddl]& & \\
 & & &*\txt<3pc>{  2  \\ 2 3 \\1 3 4}\ar^{\pi_2}[d] \\
*\txt<3pc>{  3  \\ 2 3 \\1 2 3}\ar_{\pi_2}[d]&*\txt<3pc>{  3  \\ 1 3 \\1 3 4}\ar^{\pi_3}[dl]\ar^{\pi_1}[dr]& & *\txt<3pc>{  2  \\ 1 3 \\1 3 4}\ar^{\pi_3}[d]\ar_{\pi_1}[dl]\\
*\txt<3pc>{  3  \\ 1 3 \\1 2 3}\ar_{\pi_1}[drr]& & *\txt<3pc>{  1  \\ 1 3 \\1 3 4}\ar_{\pi_3}[d]&*\txt<3pc>{  2  \\ 1 3 \\1 2 3}\ar^{\pi_1}[dl] \\
 &x= & *\txt<3pc>{  1  \\ 1 3 \\1 2 3}& 
}
$$
\caption{
\label{f1}
An interval of weak order in $\MT_4$ that is not a lattice, and a subinterval within it.}
\end{figure}
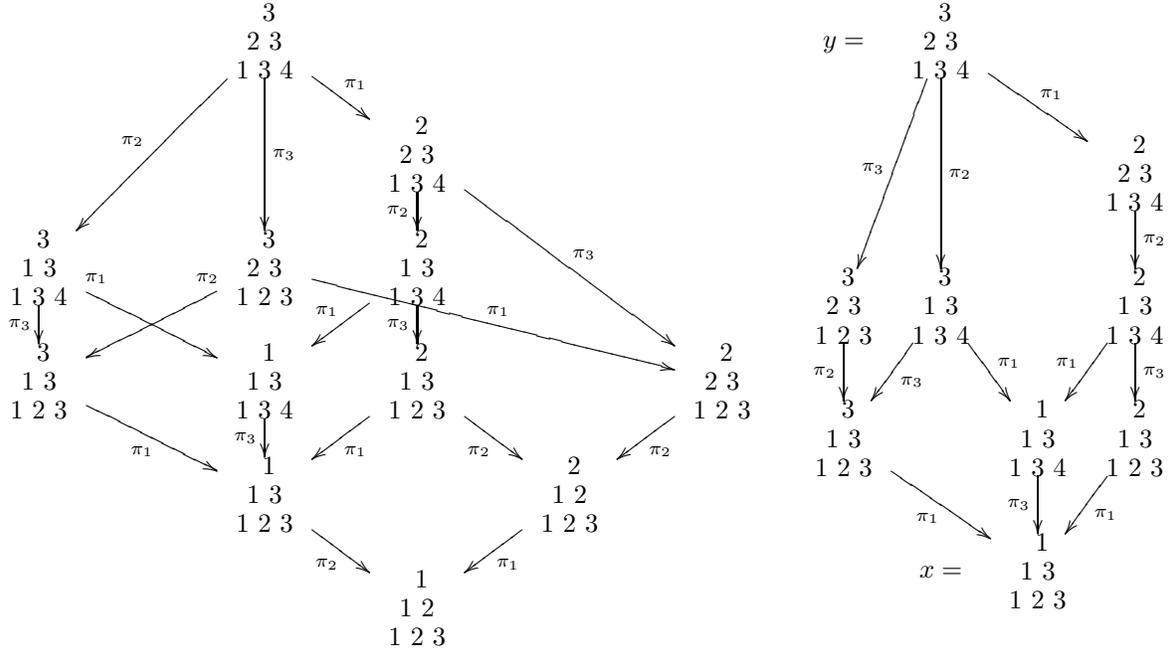

Alternating sign matrices $A = (a_{ij})$ have a well-established \emph{inversion number} 
$
\inv{A}:=\sum_{i<k, \,\, j > \ell} a_{ij} a_{k\ell},
$
introduced by Mills, Robbins and Rumsey~\cite[p344]{mills1983alternating}, 
which generalizes the rank function $\#\Inv(w)$ for $(\symm_n, <_W)$
of permutations.  However, it is not clear that it relates to chains
in the weak order $(\MT_n, <_W)$.
For example, one might hope that the length of the shortest saturated chain 
from $\hat{0}$ to $T$ in weak order might correspond to the inversion number
of the alternating sign matrix of $T$.  However, Roger Behrend noted that this fails for the first time in $\MT_4$, where one can check that
$$
T = \xymatrix@R=.5in@C=.01in{*\txt<4pc>{   3   \\  2 4  \\ 1 3 4}}
\quad \leftrightarrow \quad
A = \left[\begin{smallmatrix}
0 & 0 & + & 0\\
0 & + & - & +\\
+ & - & + & 0\\
0 & + & 0 & 0
\end{smallmatrix}
\right], \quad \inv{A} = 5
$$
but the shortest saturated chain from $\hat 0$ to $T$ has length 4.
Additionally, in $\MT_5$ one can check that 
$$
T= \xymatrix@R=.5in@C=.01in{*\txt<4pc>{   3   \\  3 4  \\ 1 4 5 \\1 2 4 5}
} \quad \leftrightarrow \quad
A=\left[
\begin{smallmatrix}
0&0&+&0&0\\
0&0&0&+&0\\
+&0&-&0&+\\
0&+&0&0&0\\
0&0&+&0&0
\end{smallmatrix}
\right], \quad \inv{A}=5
$$
but all saturated chains in the weak order 
from $\hat{0}$ to $T$ have length at least $6$.

\begin{question}
Is there a generalization of the notion of the (left-)inversion set $\Inv(w)$
for permutations to an inversion set $\Inv(T)$ for monotone triangles,
encoding the weak order $(\MT_n,<_W)$ via inclusion, that is,
$T <_W T'$ if and only if $\Inv(T) \subset \Inv(T')$?
\end{question}

In spite of some of the above shortcomings, the M\"obius function and homotopy type of
open intervals in $(\MT_n,<_W)$ may be just as simple to describe as for weak order on $\symm_n$.

\begin{conjecture}
\label{mobius-conjecture}
For two monotone triangles $T' \leq_W T$, 
the order complex $\Delta(T',T)$ of their open interval in $<_W$ is contractible unless
$T'=T \bubble \pi_{w_0(J)}$ for some  $J \subset \Des(T)$, namely,
$J:=\{m: T'_m \neq T_m\}$,
in which case $\Delta(T',T)$ is homotopy
equivalent to a $(\# J-2)$-dimensional sphere.
\end{conjecture}

Conjecture~\ref{mobius-conjecture} would imply that $\mu(T',T)=0$ in the contractible
case, and $(-1)^{\# J}$ when $T'=T \bubble \pi_{w_0(J)}$.

\begin{example} \rm \ \\
An interesting example is
the non-lattice lower interval $[\hat{0},y]$ on the left in Figure~\ref{f1},
which has the order complex $\Delta(\hat{0},y)$ of its open interval 
homotopy equivalent to a $1$-sphere (circle).  Meanwhile,
its subinterval $[x,y]$  shown to its right has $\Delta(x,y)$ contractible.
\end{example}



\begin{thebibliography}{99}

\bibitem{bjorner1980shellable}
A. Bj\"orner,
Shellable and {C}ohen-{M}acaulay partially ordered sets,
{\it Trans. AMS} {\bf 260} (1980), 159--183.

\bibitem{bjorner-orderings}
\bysame,
Orderings of Coxeter groups. 
Combinatorics and algebra (Boulder, Colo., 1983),
{\it Contemp. Math.} {\bf 34}, Amer. Math. Soc., Providence, RI, 1984.

\bibitem{bjorner1984some}
\bysame, 
Some combinatorial and algebraic properties of {C}oxeter complexes and 
{T}its buildings,
Adv. Math. {\bf 52} (1984), 173--212.

\bibitem{Bjorner-CW}
\bysame, 
Posets, regular CW complexes and Bruhat order. 
{\it European J. Combin.} {\bf 5} (1984), 7--16. 


\bibitem{bjornerbrenti}
A. Bj\"orner and F. Brenti,
Combinatorics of Coxeter groups. 
{\it Graduate Texts in Mathematics} {\bf 231}. Springer, New York, 2005.


\bibitem{bressoud1999proofs}
D.M Bressoud,
Proofs and {C}onfirmations: The {S}tory of the {A}lternating-{S}ign {M}atrix {C}onjecture,
Cambridge University Press, 1999.


\bibitem{BrubakerBumpFriedberg}
B. Brubaker, D. Bump, and S. Friedberg, 
Eisenstein series, crystals, and ice. 
{\it Notices Amer. Math. Soc.} {\bf 58} (2011), 1563--1571.

\bibitem{cheballah2015hopf}
Hayat Cheballah, Samuele Giraudo and R{\'e}mi Maurice,
Hopf algebra structure on packed square matrices,
{\it J. of Combin. Theory Ser. A} {\bf 133} (2015), 139--182.




\bibitem{Gessel} 
I.M Gessel,
Multipartite P-partitions and inner products of skew Schur functions. 
Combinatorics and algebra (Boulder, Colo., 1983),
{\it Contemp. Math.} {\bf 34}, Amer. Math. Soc., Providence, RI, 1984.


\bibitem{lascoux1996treillis}
 A. Lascoux and M.-P. Sch\"utzenberger,
 Treillis et bases des groupes de {C}oxeter,
{\it Elec. J. Combinatorics} {\bf 3} (1996), paper R27.
 
\bibitem{malvenuto1995duality}
C. Malvenuto and C. Reutenauer,
Duality between quasi-symmetric functions and the {S}olomon descent algebra,
{\it J. Algebra} {\bf 177} (1995), 967--982
 

\bibitem{mills1983alternating}
W.H. Mills, D.P. Robbins, and H. Rumsey,
 Alternating sign matrices and descending plane partitions,
 {\it   J. Combin. Theory Ser. A} {\bf 34} (1983), 340--359.
 
 
\bibitem{Norton}
P.N. Norton,
0-Hecke algebras. 
{\it J. Austral. Math. Soc. Ser. A} {\bf 27} (1979), 337--357. 

\bibitem{Petersen}
T.K. Petersen,
Eulerian numbers. 
{\it Birkh\"auser Advanced Texts: Basel Textbooks}.
Birkh\"auser/Springer, New York, 2015.



\bibitem{Stanley-commalg}
R.P. Stanley,
Combinatorics and commutative algebra (2nd ed.)
{\it Progress in Mathematics} {\bf 41}.  Birkh\"auser, Boston MA, 1996.

\bibitem{Stanley-EC2}
\bysame,
Enumerative combinatorics. Vol. 2. 
{\it Cambridge Studies in Advanced Mathematics} {\bf 62}. 
Cambridge University Press, Cambridge, 1999. 


\bibitem{terwilliger2017poset}
P. Terwilliger,
A poset {$\Phi_n$} whose maximal chains are in bijection with the $n \times n$ alternating sign matrices,
{\it Lin. Alg. and its Applications} {\bf 554} (2018), 79--85.







\end{thebibliography}
\end{document}